\documentclass[preprint,12pt,3p]{elsarticle}

\usepackage{color}
\usepackage[colorlinks=true]{hyperref}

\usepackage{amsmath,amssymb,amsthm,amscd}
\numberwithin{equation}{section}

\newtheorem{theorem}{Theorem}[section]

\newtheorem{proposition}{Proposition}[section]
\newtheorem{lemma}{Lemma}[section]

\newtheorem{corollary}{Corollary}[section]
\newtheorem{remark}{Remark}[section]
\newtheorem{ex}{Example}[section]

\journal{Elsevier}

\begin{document}

\begin{frontmatter}

\title{Estimates of Dirichlet Eigenvalues for a Class of \\ Sub-elliptic Operators\tnoteref{label0}}
\tnotetext[label0]{This work is supported by National Natural Science Foundation of China (Grants No. 11631011 and 11626251)}
\author[label1]{Hua Chen\corref{cor1}}
\ead{chenhua@whu.edu.cn}
\author[label1]{Hongge Chen}
\ead{hongge\_chen@whu.edu.cn}
\address[label1]{School of Mathematics and Statistics, Wuhan University, Wuhan 430072, China}

\cortext[cor1]{corresponding author}

\begin{abstract}
Let $\Omega$ be a bounded connected open subset in $\mathbb{R}^n$
with  smooth boundary $\partial\Omega$. Suppose that we have  a
system of real smooth vector fields $X=(X_{1},X_{2},$
$\cdots,X_{m})$ defined on a neighborhood of $\overline{\Omega}$
that satisfies the H\"{o}rmander's condition. Suppose further that
$\partial\Omega$ is non-characteristic with respect to $X$. For a
self-adjoint sub-elliptic operator $\triangle_{X}=
-\sum_{i=1}^{m}X_{i}^{*} X_i$ on $\Omega$, we denote its $k^{th}$
Dirichlet eigenvalue by $\lambda_k$. We will provide  an uniform
upper bound for the sub-elliptic Dirichlet heat kernel.  We will
also give an explicit sharp lower bound estimate for $\lambda_{k}$,
which has a polynomially growth in $k$ of the order related to the
generalized M\'{e}tivier index. We will establish an explicit
asymptotic formula of $\lambda_{k}$ that generalizes the
M\'{e}tivier's results in 1976. Our asymptotic formula shows  that
under a certain condition, our lower bound estimate for
$\lambda_{k}$ is optimal in terms of the growth of $k$. Moreover,
the upper bound estimate of the Dirichlet eigenvalues for general
sub-elliptic operators will  also be given, which, in a certain
sense, has the optimal growth order.
\end{abstract}


\end{frontmatter}


\section{Introduction and Main Results}
For $n\geq 2$, let $X=(X_{1},X_{2},\cdots,X_{m})$ be the system of
$C^{\infty}$ real vector fields  defined over a domain $W$ in
$\mathbb{R}^{n}$. For our study here, the essential hypothesis  is
the following H\"{o}rmander's condition: (cf. \cite{hormander1967})

(H): $X_{1},X_{2},\ldots,X_{m}$ together with their commutators up to a certain fixed length span the tangent space at each point of $W$.

We introduce the following weighted Sobolev spaces (cf.
\cite{xu1992}) associated with $X$,
 \[ H_{X}^{1}(W)=\{u\in L^{2}(W)~|~X_{j}u\in L^{2}(W), j=1,\cdots,m\}. \]
Then $H_{X}^{1}(W)$ is a Hilbert space endowed with norm $\|u\|^2_{H^{1}_{X}(W)}=\|u\|_{L^2(W)}^2+\|Xu\|_{L^2(W)}^2$, where  $\|Xu\|_{L^2(W)}^2=\sum_{j=1}^{m}\|X_{j}u\|_{L^2(W)}^2$. \par
 Let  $\Omega\subset\subset W$ be a bounded connected
  open subset with $C^{\infty}$ boundary and the boundary $\partial\Omega$ is assumed to be
  non-characteristic for $X$ (i.e. for any $x_{0}\in\partial\Omega$, there exists at least one vector
  field $X_{j_{0}}~(1\leq j_0\leq m)$, such that $X_{j_{0}}(x_{0})\notin T_{x_{0}}(\partial\Omega)$).
  Then the space $H_{X,0}^{1}(\Omega)$ being the closure of $C_{0}^{\infty}(\Omega)$ in $H_{X}^{1}(W)$
  is well-defined, and is also a Hilbert space. Clearly, the vector fields in $X$ satisfy the
   condition (H) on $\overline{\Omega}$. Hence there is an integer $Q$
   such that the vector fields $X_{1},X_{2},\ldots,X_{m}$ together with their commutators
   of length at most $Q$ span the tangent space $T_{x}(W)$ at each point  $x\in \overline{\Omega}$.
    Recall that $Q$ is called the  H\"{o}rmander's index of $\overline{\Omega}$ with respect to $X$.
     We say that the vector fields $X$ are finitely degenerate if $2\leq Q<+\infty$. \par

Consider the following H\"{o}rmander type operator
\[ \triangle_{X}:=-\sum_{i=1}^{m}X_{i}^{*}X_{i},\]
where $X_{i}^{*}$ is the formal adjoint of $X_{i}$. (In general,
$X_{i}^{*}=-X_{i}-\text{div}X_{i} $, where $\text{div}X_{i}$ is the
divergence of $X_{i}$.) Since  $-\triangle_{X}$ is symmetric on
$C_{0}^{\infty}(\Omega)$, it is easy to show that, after
self-adjoint extension, $-\triangle_{X}$ can be uniquely extended to
a positive unbounded self-adjoint operator on the domain
$D(\triangle_{X})=\{u\in H_{X,0}^{1}(\Omega)|\triangle_{X}u\in
L^2(\Omega) \}$. \par
 In this paper, we mainly focus on the following Dirichlet eigenvalue problem in $H_{X,0}^{1}(\Omega)$,
    \begin{equation}\label{1-1}
  \left\{
         \begin{array}{ll}
           -\triangle_{X}u=\lambda u, & \hbox{in $\Omega$;} \\[3mm]
           u=0, & \hbox{on $\partial\Omega$.}
         \end{array}
       \right.
\end{equation}
From the condition (H), we know that the sub-elliptic self-adjoint
operator $-\triangle_{X}$ defined on $D(\triangle_{X})$ has discrete
eigenvalues
$0<\lambda_1\leq\lambda_2\leq\cdots\leq\lambda_{k-1}\leq\lambda_k\leq\cdots$,
and $\lambda_{k}\to +\infty $ as $k\to +\infty$.

When $X=(\partial_{x_{1}},\cdots,\partial_{x_{n}})$, $\triangle_{X}$
is the standard Laplacian $\triangle$. In this classical case, there
have been extensive studies on  the estimate of its eigenvalues.
Here we mention the work done in
\cite{chenqingmin1,chenqingmin2,Kroger,Yau1983,Polya1961, Weyl1912}
as well as the references therein.

If the vector fields in $X$ satisfy the condition (H) on
$\overline{\Omega}$ with H\"{o}rmander index $Q\geq 2$, M\'{e}tivier
\cite{Metivier1976} initiated the study on   the asymptotic behavior
of the eigenvalues under an extra  assumption on $X$:

For each $x\in \overline{\Omega}$, let $V_{j}(x)~ (1\leq j\leq Q)$
be the subspaces of the tangent space at $x$  spanned by all
commutators of $X_{1},\ldots,X_{m}$ with length at most $ j$.
M\'{e}tivier made the following assumption:

(M): For each $x\in\overline{\Omega}$, $\dim V_{j}(x)$ is a
constant (denoted by $\nu_{j}$) in a neighborhood of $x$.

\par
 Under the above additional hypothesis (M), M\'{e}tivier in \cite{Metivier1976} proved the following
 asymptotic expression for the sub-elliptic Dirichlet eigenvalue $\lambda_{k}$,
\begin{equation}\label{1-2}
  \lambda_{k}\sim \left(\int_{\Omega}\gamma(x)dx\right)^{-\frac{2}{\nu}}\cdot k^{\frac{2}{\nu}}~\mbox{ as } k\to+\infty,
\end{equation}
where $\gamma(x)$ is a positive continuous function on $\Omega$. The index $\nu$ is defined as
\begin{equation}\label{1-3}
   \nu:=\sum_{j=1}^{Q}j(\nu_{j}-\nu_{j-1}),\qquad \nu_{0}:=0,
\end{equation}
which is called the  M\'{e}tivier index of $\Omega$ (here $\nu$ is
also called the Hausdorff dimension of $\Omega$ related to the
sub-elliptic metric induced by the vector fields $X$).

The asymptotic formula \eqref{1-2} fails to hold  for general
H\"{o}rmander vector fields not satisfying the M\'{e}tivier
condition. To our best knowledge, there is little information in
literature about the explicit asymptotic behavior of Dirichlet
eigenvalues for  general sub-elliptic operators which only satisfy
H\"{o}rmander's condition (H). Recently, in the case of
$X_j^*=-X_j$, Chen and Luo in \cite{chenluo} estimated the lower
bound of $\lambda_k$ for the self-adjoint sum of square operator
$L=-\sum_{j=1}^{m}X_{j}^2$. They proved that
\begin{equation}\label{1-4}
  \sum_{j=1}^{k}\lambda_{j}\geq C_{0}\cdot k^{1+\frac{2}{nQ}}~~ \mbox{ for any } k\geq 1,
\end{equation}
where $C_{0}$ is a positive constant related to $X$ and $\Omega$.
 Consequently, \eqref{1-4} implies $\lambda_k\geq C_{0}\cdot k^{\frac{2}{nQ}}$.\par
From $\eqref{1-3}$, we can deduce that $n+Q-1\leq \nu\leq nQ$, and
actually $\nu=nQ$ if and only if $Q=1$. It can be  seen that if $X$
satisfy the condition (M) with H\"{o}rmander index $Q>1$, the growth
order  for $\lambda_{k}$ in \eqref{1-4} is $\frac{2}{Qn}$, which is
smaller than the one in M\'{e}tivier's asymptotic formula
\eqref{1-2}. This shows  that Chen and Luo's lower bound estimate
\eqref{1-4} is not optimal  under the condition (M).

There are many results under the M\'{e}tivier's condition (M), such
as sub-elliptic estimates and function spaces on nilpotent Lie
groups, Sobolev inequalities, Harnack inequality and heat kernel
estimates on nilpotent Lie groups (cf. \cite{Folland, Varopoulos}).
Parallel to the classical Laplacian $\triangle$ in $\mathbb{R}^n$,
the Kohn Laplacian operator $\triangle_{\mathbb{H}}$ induced by left
invariant vector fields on Heisenberg group $(\mathbb{H}_{n},\circ)$
is a sub-elliptic operator which plays an important role in physics.
In 1994, Hansson and Laptev \cite{AM1994} gave a precise lower
bounds of Dirichlet eigenvalues $\lambda_k$ for the Kohn Laplacian
operator $\triangle_{\mathbb{H}}$. The M\'etivier's condition posses
a strong  restriction on the vector fields $X$ satisfying
H\"ormander's condition, under which the Lie algebra generated by
the vector fields $X_{1},X_{2},\ldots,X_{m}$ takes a constant
structure and the vector fields can be well approximated by some
homogeneous left invariant vector fields defined on the
corresponding Carnot group (cf. \cite{stein1976}). In this paper, we
will deal with general self-adjoint H\"{o}rmander operators
$-\triangle_X=\sum_{i=1}^{m}X_{i}^{*}X_{i}$ without the restriction
of M\'{e}tivier's condition (M). A main purpose is to establish a
sharp lower bound of the Dirichlet eigenvalue $\lambda_k$ for the
sub-elliptic operator $-\triangle_{X}$. Furthermore, we construct an
asymptotic formula for $\lambda_{k}$ which is a generalization of
M\'{e}tivier's result \eqref{1-2}. In fact,  M\'{e}tivier's
condition (M) is just a sufficient condition for this generalized
asymptotic formula. Our discussion  below demonstrates that the
M\'{e}tivier's condition (M) can be relaxed to a weak condition
which is now the  necessary and sufficient condition for the
asymptotic formula of $\lambda_{k}$ being satisfied. Also, under
this weak condition, the asymptotic formula shows that our lower
bound for $\lambda_{k}$ is optimal in terms of the order on $k$.

In this paper, the general H\"{o}rmander vector fields $X$ need not
necessary to satisfy the M\'{e}tivier's condition (M). Therefore,  we need to
introduce the following generalized M\'{e}tivier's index which is
also called the non-isotropic dimension of $\Omega$ related to $X$
(cf. \cite{chenluo, Morbidelli2000, Yung2015}). With the same
notations as before, we denote here $\nu_{j}(x)=\dim V_{j}(x)$ and
then $\nu(x)$, the pointwise homogeneous dimension at $x$, is given
by
\begin{equation}\label{1-5}
  \nu(x):=\sum_{j=1}^{Q}j(\nu_{j}(x)-\nu_{j-1}(x)),\qquad \nu_{0}(x):=0.
\end{equation}
Then we define
\begin{equation}\label{1-6}
  \tilde{\nu}:=\max_{x\in\overline{\Omega}} \nu(x)
\end{equation}
as the generalized M\'{e}tivier index of $\Omega$. Observe that $n+Q-1\leq \tilde{\nu}< nQ$ for $Q>1$, and  $\tilde{\nu}=\nu$ if the M\'{e}tivier's condition (M) is satisfied.

In \cite{chenluo}, Chen and Luo considered  the Grushin vector fields
 $X=(\partial_{x_{1}},\cdots,\partial_{x_{n-1}},x_{1}^{l}\partial_{x_{n}})$  defined
 in $\mathbb{R}^n$ ($n\geq2$, $l\in \mathbb{N}^{+}$). The domain $\Omega$ is assumed to be a
  bounded connected
 open subset with smooth non-characteristic boundary for $X$
 and $\Omega\cap\{x_{1}=0\}\neq \varnothing$. In this case, the M\'{e}tivier's condition (M) is not satisfied.
  However, the vector fields $X$ are finitely degenerate with $Q=l+1\geq 2$ and the generalized M\'{e}tivier
  index $\tilde{\nu}=n+Q-1=n+l$. Then the Chen and Luo's results in \cite{chenluo}
   gave a sharp lower bound estimates for Dirichlet eigenvalues of
   $-\triangle_{X}$, i.e. $\lambda_{k}\geq c_{1} k^{\frac{2}{\tilde{\nu}}}$.
   In \cite{CCD}, the authors further extended this result to more general Grushin type operators.\par
   We now  return
    to our general consideration. Our first goal is to show that the above sharp lower bound is
    also hold for general sub-elliptic operator $\triangle_{X}$. The key ingredient of
    our argument is to establish the following uniform upper bound for the Dirichlet heat kernel of
     sub-elliptic operator $\triangle_{X}$:
\begin{theorem}\label{thm1}
Let $X=(X_{1},X_{2},\cdots,X_{m})$ be $C^{\infty}$ real vector fields defined on a connected open domain
  $W\subset\mathbb{R}^n$, which satisfy the condition (H) in $W$.  Assume that $\Omega\subset\subset W$ is a bounded connected open subset, and $\partial\Omega$ is smooth  and non-characteristic for $X$. If the H\"{o}rmander index $Q\geq 2$, then the self-adjoint sub-elliptic operator $\triangle_{X}=-\sum_{i=1}^{m}X_{i}^{*}X_{i} $ has a positive smooth
Dirichlet heat kernel $h_D(x,y,t)\in C^{\infty}(\Omega\times\Omega\times(0,+\infty)) $, which satisfies the following  uniform upper bound estimate
\begin{equation}\label{1-7}
  h_D(x,x,t)\leq \frac{C}{t^{\frac{\tilde{\nu}}{2}}} ~~\mbox{  for all }~~(x,t)\in \Omega\times(0,+\infty),
\end{equation}
where $\tilde{\nu}$ is the generalized M\'{e}tivier index of $X$ on $\Omega$, and $C$ is a positive constant depending on $X$ and $\Omega$.
\end{theorem}
From Theorem \ref{thm1}, we can deduce the following sharp lower
bound estimate of $\lambda_{k}$ for the sub-elliptic Dirichlet
problem \eqref{1-1}.
\begin{theorem}\label{thm2}
Suppose that $X=(X_{1},X_{2},\cdots,X_{m})$ satisfy the same conditions of Theorem \ref{thm1}.  Then for any $k\geq1$, we have
\begin{equation}\label{1-8}
 \sum_{j=1}^{k}\lambda_{j}\geq C_{1}\cdot k^{1+\frac{2}{\tilde{\nu}}},
\end{equation}
where $\tilde{\nu}$ is the generalized M\'{e}tivier index of $X$ on $\Omega$,
 $C_{1}=(C\mbox{e}|\Omega|)^{-\frac{2}{\tilde{\nu}}}$ is a positive constant depending on the volume of
 $\Omega$ and $\tilde{\nu}$, and $C$ is the same constant as in \eqref{1-7}.
\end{theorem}
Furthermore, we obtain the following asymptotic formula for the sub-elliptic Dirichlet eigenvalues $\lambda_{k}$.
\begin{theorem}\label{thm3}
Suppose that $X=(X_{1},X_{2},\cdots,X_{m})$ satisfy the same
conditions as Theorem \ref{thm1}. Then there exists a non-negative
measurable function $\gamma_{0}$ on $\overline{\Omega}$ with
$\gamma_{0}(x)>0$ for all $x\in \Omega$ such that
\begin{equation}
\label{1-9}
\lim_{\lambda\to +\infty} \lambda^{-\frac{\tilde{\nu}}{2}}N(\lambda)=\frac{1}{\Gamma\left(\frac{\tilde{\nu}}{2}+1 \right)}\cdot \int_{H}\gamma_{0}(x)dx.
\end{equation}
Here $H:=\{x\in \Omega~|~\nu(x)=\tilde{\nu}\}$ is a subset of $\Omega$, and $N(\lambda):=\#\{k~|~0<\lambda_{k}\leq \lambda\}$. Moreover, we can deduce that
\begin{itemize}
  \item If $|H|>0$, we have
\begin{equation}\label{1-10}
\lambda_{k}=\left(\frac{\Gamma\left(\frac{\tilde{\nu}}{2}+1\right)}{\int_{H}\gamma_{0}(x)dx}\right)^{\frac{2}{\tilde{\nu}}}\cdot
k^{\frac{2}{\tilde{\nu}}}+o(k^{\frac{2}{\tilde{\nu}}})~\mbox{ as } k\to +\infty;
\end{equation}
  \item If $|H|=0$, then we have
\begin{equation}\label{1-11}
  \lim_{k\to+\infty}\frac{k^{\frac{2}{\tilde{\nu}}}}{\lambda_{k}}=0.
\end{equation}
\end{itemize}
\end{theorem}

The results of Theorem \ref{thm3} have the following obvious corollary.

\begin{corollary}\label{cor4}
For the Dirichlet eigenvalues $\lambda_k$ of sub-elliptic operator
$-\triangle_{X}$ on $\Omega$. Also
  $\lambda_k \approx
k^{\frac{2}{\tilde{\nu}}}$ holds as $k\to +\infty$ if and only if $|H|>0$.
\end{corollary}

\begin{remark}
We mention that from the Theorem \ref{thm3}  if $H=\{x\in
\Omega~|~\nu(x)=\tilde\nu\}$ has a positive
 measure, the lower bound \eqref{1-8} for $\lambda_{k}$ in Theorem \ref{thm2} is optimal in terms  of the
 order on $k$. In particular, when  M\'{e}tivier's condition (M) is satisfied, we know that $H=\Omega$
  and the condition $|H|>0$ is certainly satisfied. In this case, our asymptotic
 formula \eqref{1-10}  coincides with  M\'{e}tivier's asymptotic estimate \eqref{1-2}. If $H$ has zero measure,
 then the result of Theorem \ref{thm3} implies that our lower bound estimate \eqref{1-8} for the eigenvalue
 $\lambda_{k}$ is not optimal, for $\lambda_{k}^{-1}=o(k^{-\frac{2}{\tilde{\nu}}})$ as $k\to+\infty$.
\end{remark}

\begin{remark}
The result of Corollary \ref{cor4}  has the following geometric
meaning: Under the condition $|H|>0$, the non-isotropic dimension
$\tilde\nu$ of $\Omega$ related to $X$ will be a spectral invariant.
\end{remark}

For upper bounds of Dirichlet eigenvalues $\lambda_{k}$
 for sub-elliptic operator $-\triangle_{X}$, we have the following result.

\begin{theorem}\label{thm4}
Assume that the real smooth vector fields
$X=(X_{1},X_{2},\cdots,X_{m})$ satisfy the same conditions as in
Theorem \ref{thm1}. Then for any $k\geq 1$ and the $k^{th}$
Dirichlet eigenvalue $\lambda_{k}$ for the sub-elliptic operator
$-\triangle_{X}$, we have
\begin{equation}\label{1-12}
  \lambda_{k}\leq \tilde{C}\cdot (k-1)^{\frac{2}{n}}+\lambda_{1}~~\mbox{for all}~k\geq 1,
\end{equation}
where $\tilde{C}>0$ is a constant depending on $X$ and $\Omega$.
\end{theorem}

It is well-known that, in the non-degenerate case, the eigenvalues
$\lambda_{k}$ of Dirichlet Laplacian have asymptotic behavior
$\lambda_{k}\approx k^{\frac{2}{n}}$ as $k\to +\infty$. Thus, the
result in \eqref{1-12} means that the upper bounds of Dirichlet
eigenvalues $\lambda_{k}$ of $-\triangle_{X}$ have the same order in
$k$ with that in the non-degenerate case. If the H\"ormander index
$Q>1$, we have $\tilde{\nu}>n$. Then from  Corollary \ref{cor4}
above, we know that the upper bounds \eqref{1-12} is not optimal in
the case of $|H|>0$. However, the following result demonstrates that
the result of the upper bounds \eqref{1-12} cannot be improved in
general in case $|H|=0$. To be more detailed, we  introduce the
following condition:

(A): We say that the vector fields $X=(X_{1},X_{2},\cdots,X_{m})$ satisfy assumption (A) on $\Omega$ if
   \begin{equation}\label{1-13}
   \int_{\Omega}\frac{dx}{\sum|\det(Y_{i_{1}},Y_{i_{2}},\cdots,Y_{i_{n}})(x)|}<+\infty.
   \end{equation}
   Here the sum is over all $n$-combinations $(Y_{i_{1}},Y_{i_{2}},\cdots,Y_{i_{n}})$ of the set
    $\{X_{1},X_{2},\cdots,X_{m}\}$.

Actually, we can deduce $|H|=0$ from the condition (A). In fact, for
each $x\in H$, since $Q\geq 2$, we have $\nu_{1}(x)=\dim
V_{1}(x)<n$. This is because that if $\nu_{1}(x)=n$, then
$V_{1}(x)=V_{2}(x)=\cdots=V_{Q}(x)=T_{x}(W)$, which implies that
$\nu(x)=\sum_{j=1}^{Q}j(\nu_{j}(x)- \nu_{j-1}(x))=n$, but
$\tilde{\nu}\geq n+Q-1\geq n+1>\nu(x)$, that means $x\notin H$. Thus
we introduce the set $E$ by $E:=\{x\in \Omega
|\sum|\det(Y_{i_{1}},Y_{i_{2}},\cdots,Y_{i_{n}})(x)|=0\}$. Then for
any $x\in H$, the fact $\nu_{1}(x)=\dim V_{1}(x)<n$ implies that
$\sum|\det(Y_{i_{1}},Y_{i_{2}}, \cdots,Y_{i_{n}})(x)|=0$, where the
sum is taken over all  $n$-combinations
$(Y_{i_{1}},Y_{i_{2}},\cdots,Y_{i_{n}})$ of the set
$\{X_{1},X_{2},\cdots,X_{m}\}$. Hence we have $H\subset E$.

 On the other hand, if we write $g(x)=\sum|\det(Y_{i_{1}},Y_{i_{2}},\cdots,Y_{i_{n}})(x)|\geq 0$,
  where the sum is taken over all  $n$-combinations $(Y_{i_{1}},Y_{i_{2}},\cdots,Y_{i_{n}})$ of
  the set $\{X_{1},X_{2},\cdots,X_{m}\}$. Thus if we let $A= \int_{\Omega}\frac{dx}{\sum|\det(Y_{i_{1}},
  Y_{i_{2}},\cdots,Y_{i_{n}})(x)|}<+\infty$, then \eqref{1-13} implies that
  the set $E_{n}=\{x\in \Omega|\frac{1}{g(x)} \geq n \} $ satisfies $|E_{n}|\leq \frac{A}{n}$
  for each $n\in \mathbb{N}^{+}$. Observe that $E_{n}=\{x\in \Omega|\frac{1}{g(x)} \geq n \}=\{x\in
   \Omega|0\leq g(x)\leq \frac{1}{n}\}$ and $E_{n+1}\subset E_{n}$. We then
    have $E=\{x\in \Omega|g(x)=0 \}=\cap_{n=1}^{\infty}E_{n}$. Therefore,
     $|E|=\lim_{n\to \infty}|E_{n}|=0$. Since $H\subset E$, we obtain $|H|=0$.\par

Our next result is stated as follows.
\begin{theorem}\label{thm5}
If the real smooth vector fields $X=(X_{1},X_{2},\cdots,X_{m})$ satisfy the same conditions of Theorem \ref{thm1} and assumption $(A)$ on $\Omega$, then we have
\begin{equation}\label{1-14}
  \sum_{i=1}^{k}\lambda_{i}\geq C\cdot k^{1+\frac{2}{n}}~~\mbox{for all}~~ k\geq 1.
\end{equation}
Here the constant $C>0$ is independent of $k$, and $\lambda_{i}$ is the $i^{th}$ Dirichlet eigenvalue of problem \eqref{1-1}.
\end{theorem}

\begin{remark}
The conclusion of Theorem \ref{thm5} implies that, under condition
(A), the Dirichlet eigenvalues $\lambda_{k}$ for a degenerate
elliptic operator $-\triangle_X$ will have the same asymptotic
behavior with the non-degenerate Laplacian case: $\lambda_{k}\approx
k^{\frac{2}{n}}$ as $k\to +\infty$. Also in this case, the upper
bound \eqref{1-12} for $\lambda_{k}$ is optimal in terms of  the
growth  order in $k$.
\end{remark}

\begin{remark}
One important study where the system $X$ appears in application is
when one studies the CR vector fields of CR manifolds. For
simplicity, we let $M$ be a smooth real hypersurface in a complex
Euclidean space ${\mathbb C}^n$ with $n\ge 2$ defined by $\rho=0$.
Write  $(z_1,\cdots, z_n)$ for the coordinates of ${\mathbb C}^n$.
Assume without loss of generality that
$\rho_{z_n}:=\frac{\partial{\rho}}{\partial{z_n}}\not =0$ along $M$.
Then $L_j=\frac{\partial{
}}{\partial{z_j}}-\frac{{\rho_{z_j}}}{{\rho_{z_n}}}\frac{\partial{\
}}{\partial{z_n}}$ for $j=1,\cdots, n-1$ form a basis of CR vector
fields along $M$. Let $X_j=\hbox{Re}(L_j)$ and
$X_{j+n-1}=\hbox{Im}(L_j)$. Then the system $X=\{X_1,\cdots,
X_{2n-2}\}$ satisfies the H\"ormander condition if and only if $M$
is of finite type in the sense of Bloom-Graham that is equivalent to
the geometric condition that there is no complex hypersurface
contained in $M$ (see the book of Baouendi-Ebenfelt-Rothschild
\cite{BER} for related references). When $M$ is Levi non-degenerate,
then the H\"ormander index of $X$ is always $2$ at each point along
$M$ and thus the M\'{e}tivier condition holds. The other situation where
the M\'{e}tivier condition holds is when $M$ has  uniform finite
non-degeneracy (see the work of Baounendi-Huang-Rothschild
\cite{BHR} for definition and  many examples of this type
hypersurfaces). For instance, this is the case when $M\subset
{\mathbb C}^3$ is the Freeman cone  defined by $z_1^2+z_2^2=z_3^2$.
In general, the M\'{e}tivier condition is rarely satisfied for $X$ with
such a geometric background. The generalized M\'{e}tivier index is
associated with the degeneracy of the Levi-form along $M$. It is two
if and only if the point is a Levi non-degenerate point along at
least one CR direction and is at least three otherwise. The
H\"ormander sub-elliptic Laplacian associated with $X$ is more or
less the Kohn's sub-Laplacian operator of $M$. There have been much
work done to study the spectral theory in the strongly pseudo-convex
case (see \cite{BGS}). Our result in the present paper may shed the
light for $M$ being weakly pseudo-convex but of finite type where
much less is known. We hope to come back to such an application in a
future work.
\end{remark}

\begin{remark}
Some other results on eigenvalues of hypoelliptic operators, one can see \cite{Menikoff-Sjostrand1978, Menikoff-Sjostrand1978-1, Menikoff-Sjostrand1979, Sjostrand1980, Mohamed1993} and references therein.
\end{remark}

The plan of the rest paper is as follows. In Section 2, we present some preliminaries
 including the weighted Sobolev embedding theorem, the weighted Poincar\'{e}
 inequality induced by vector fields $X$, the sub-elliptic estimates, Carnot-Carath\'{e}odory
 metric and the estimate of volume for subunit ball. In Section 3,
 we establish a supremum norm estimates of the Dirichlet eigenfunction and
 an explicit lower bound of the Dirichlet eigenvalue. In Section 4, we discuss the existence
  of the Dirichlet heat kernel for the sub-elliptic operator $\triangle_{X}$ and
  some basic properties for the fundamental solution of the degenerate heat equation.
   In Section 5, we study the diagonal asymptotic  behavior of the Dirichlet heat kernel
    for the sub-elliptic operator $\triangle_{X}$.
    The proofs of Theorem \ref{thm1}, Theorem \ref{thm2} and Theorem \ref{thm3} will be given
    in Section 6, and the proofs of Theorem \ref{thm4} and Theorem \ref{thm5} will be given in Section 7
     respectively. Finally, as applications of Theorem \ref{thm2} -- Theorem \ref{thm5},
      we shall present more related  examples in Section 8.

\section{Preliminaries}
\subsection{Some estimates on weighted Sobolev spaces.}
We start with the following weighted Sobolev embedding theorem.
\begin{proposition}[Weighted Sobolev Embedding Theorem]
 \label{pro2-1}
 Let $X=(X_{1},X_{2},\cdots,X_{m})$ be $C^{\infty}$ vector fields defined on a connected open subset
  $W$ in $\mathbb{R}^n$, which satisfy condition (H). Assume that $ \Omega\subset\subset W$ is a bounded
  open subset with  smooth boundary $\partial\Omega$ which is non-characteristic for $X$.
   Denote $\tilde{\nu}$ by the generalized M\'{e}tivier index of $X$ on $\Omega$.
   Then for $1\leq p<\tilde{\nu}$,  there exists a constant $C=C(\Omega,X)>0$, such that for all
    $u\in C^{\infty}(\overline{\Omega})$, the inequality
\begin{equation}\label{2-1}
\|u\|_{L^{q}(\Omega)}\leq C\left(\|Xu\|_{L^{p}(\Omega)}+\|u\|_{L^{p}(\Omega)}\right)
\end{equation}
holds for $q=\frac{\tilde{\nu}p}{\tilde{\nu}-p}$.
\end{proposition}
\begin{proof}
See Corollary 1 in \cite{Yung2015}.
\end{proof}
In particular, if $Q\geq 2$, then $\tilde{\nu}\geq n+Q-1\geq 3$. Putting $p=2$ into  Proposition \ref{pro2-1}, we can deduce that
\begin{equation}
\label{2-2}
\left(\int_{\Omega}|u|^{\frac{2\tilde{\nu}}{\tilde{\nu}-2}}dx\right)^{\frac{\tilde{\nu}-2}{2\tilde{\nu}}}\leq C\left(\int_{\Omega}|Xu|^2dx+\int_{\Omega}|u|^2dx \right)^{\frac{1}{2}},
\end{equation}
where $u\in H_{X,0}^{1}(\Omega)$.

\par
We also have the following weighted Poincar\'{e} inequality for the vector fields $X$.
\begin{proposition}[Weighted Poincar\'{e} Inequality]
\label{Poincare}
 Suppose that $X=(X_{1},X_{2},\cdots,X_{m})$ satisfy the same conditions as in
 Theorem \ref{thm1}. Then the first eigenvalue $\lambda_{1}$ of the Dirichlet problem \eqref{1-1}
 for $-\triangle_{X}$ is positive. Moreover, we have the following weighted Poincar\'{e} inequality
\begin{equation}\label{2-3}
  \lambda_{1}\int_{\Omega}{|u|^2dx}\leq \int_{\Omega}|Xu|^2dx,~~ \forall u\in H_{X,0}^{1}(\Omega).
  \end{equation}
\end{proposition}
\begin{proof}
We set
\[ \lambda_{1}=\inf_{\|\varphi\|_{2}=1, \varphi\in H_{X,0}^{1}(\Omega)}\|X\varphi\|^2_{L^2(\Omega)}. \]
Suppose  $\lambda_{1}=0$. Then there exists a sequence
$\{\varphi_{j}\}$ in  $H_{X,0}^{1}(\Omega)$ such that
$\|X\varphi_{j}\|_{L^2(\Omega)}\to 0$ with
$\|\varphi_{j}\|_{L^2(\Omega)}=1$. Since $H_{X,0}^{1}(\Omega)$ is
compactly embedded into $L^2(\Omega)$ (see \cite{Derridj,
Manfredini}), the variational calculus ensures that there exists
$\varphi_{0}\in H_{X,0}^{1}(\Omega)$ with
$\|\varphi_{0}\|_{L^2(\Omega)}=1$ that satisfies
$\triangle_{X}\varphi_{0}=0$ and $\|X\varphi_{0}\|_{L^2(\Omega)}=
0$. The condition (H) implies that $\triangle_{X}$ is hypo-elliptic
on $\Omega$. Meanwhile,  $\partial\Omega$ is $C^{\infty}$ and
non-characteristic for $X$. Thus, we know that $\varphi_{0}\in
C^{\infty}(\overline{\Omega})$ and $\varphi_{0}|_{\partial\Omega}=0$
(see \cite{Derridj, Kohn, Oleinik}). By  Bony's strong maximum
principle (see \cite{Bony1969, Oleinik}), we can deduce that
$\varphi_{0}$ must attain its maximum and minimum values on
$\partial\Omega$  unless $\varphi_{0}$ is a constant on
$\overline{\Omega}$. Thus we obtain $\varphi_{0}\equiv 0$, which
contradicts with $\|\varphi_{0}\|_{L^2(\Omega)}=1$. We thus proved
that $\lambda_{1}>0$.
\end{proof}

Combining \eqref{2-2} with \eqref{2-3}, we obtain the following
weighted Sobolev inequality.
\begin{proposition}[Weighted Sobolev Inequality]
\label{pro2-3} Suppose that $X=(X_{1},X_{2},\cdots,X_{m})$ satisfy
the same conditions as in Theorem \ref{thm1}. Then there exists a
constant $C=C(\Omega,X)>0$, such that for any $u\in
H_{X,0}^{1}(\Omega)$ we have
 \begin{equation}\label{2-4}
\left(\int_{\Omega}|u|^{\frac{2\tilde{\nu}}{\tilde{\nu}-2}}dx\right)^{\frac{\tilde{\nu}-2}{2\tilde{\nu}}}\leq C\left(\int_{\Omega}|Xu|^{2}dx\right)^{\frac{1}{2}}.
 \end{equation}
 \end{proposition}

Also, we need the following sub-elliptic estimates.

\begin{proposition}[Sub-elliptic estimates I]
\label{subelliptic-estimate}
Assume that $X=(X_{1},X_{2},\cdots,X_{m})$ satisfy the condition (H) on an open domain $W$ in $\mathbb{R}^n$. Then, for any open subset $\Omega\subset\subset W$, there exist constants $\epsilon_{0}>0$ and  $C>0$ such that
\begin{equation}\label{2-5}
  \|u\|_{H^{\epsilon_{0}}(\mathbb{R}^n)}^2\leq C\left(\sum_{i=1}^{m}\|X_{i}u\|_{L^2(\mathbb{R}^n)}^{2}+\|u\|_{L^2(\mathbb{R}^n)}^2\right), ~~\forall u\in H_{X,0}^{1}(\Omega).
\end{equation}
\end{proposition}
\begin{proof}
 See Theorem 17 in \cite{stein1976}.
\end{proof}
\begin{proposition}[Sub-elliptic estimates II]
\label{Pro2-5}
Suppose that $X=(X_{1},X_{2},\cdots,X_{m})$ satisfy the condition (H) on an open domain $W$ in $\mathbb{R}^n$. Let $\Omega\subset\subset W$ be an open subset and $\phi\prec \phi_1$ be nested cut-off functions with support in $\Omega$ (i.e. $\phi,~\phi_1\in C_{0}^{\infty}(\Omega)$, and $\phi_1\equiv 1$ on the support of $\phi$ ). Then there exists $\epsilon>0$ so that for every $s\geq 0$, there is a constant $C>0$ such that
\begin{equation}\label{2-6}
\|\phi u\|_{H^{s+\epsilon}(\mathbb{R}^n)}\leq C\left(\|\phi_1\triangle_{X}u\|_{H^s(\mathbb{R}^n)}+\|\phi_1 u  \|_{L^2(\mathbb{R}^n)}\right),
\end{equation}
holds for any $u\in L^2(\Omega)\cap C^{\infty}(\Omega)$.
\end{proposition}
\begin{proof}
See Theorem 17.0.1 in \cite{Nagel}, Theorem 18 in \cite{stein1976} and also refer to \cite{Kohn0}.
\end{proof}

From the Sobolev imbedding theorem we know that for $s>\frac{n}{2}$, there exists a constant $C>0$ such that
\begin{equation}\label{2-7}
  \sup_{x\in\mathbb{R}^n}|u(x)|\leq C\|u\|_{H^s(\mathbb{R}^n)}~~\mbox{for all}~~u\in H^{s}(\mathbb{R}^n).
\end{equation}
 Thus, combining \eqref{2-7} with Proposition \ref{Pro2-5}, we have following corollary.
\begin{corollary}
\label{corollary2-1}
Let $N\in \mathbb{N}^{+}$ with $N>\frac{n}{2\epsilon}$ (where $\epsilon$ was given in Proposition \ref{Pro2-5}) and $\xi(x)\in C_{0}^{\infty}(\Omega)$. If $u\in L^2(\Omega)\cap C^{\infty}(\Omega)$ and $(\triangle_{X})^{k}u\in L^2(\Omega)$ for $1\leq k\leq N$, then we have
\begin{equation}\label{2-8}
\sup_{x\in\Omega}|\xi(x)u(x)|\leq C\sum_{k=0}^{N}\|(\triangle_{X})^{k}u\|_{L^2(\Omega)}.
\end{equation}
\end{corollary}
 \begin{proof}
See Corollary 17.0.2 in \cite{Nagel}.
 \end{proof}

\subsection{Carnot-Carath\'{e}odory metric and volume of subunit ball.}
We briefly introduce some geometric properties of the metric induced
by vector fields $X$ in this part. Readers can refer to
\cite{Fefferman},\cite{Stein1985} and \cite{Morbidelli2000} in
details.

Let $X=(X_{1},X_{2},\cdots,X_{m})$ satisfy the condition (H) on a connected open set
 $V\subset\mathbb{R}^n$ with H\"{o}rmander's index $r_{0}$. Then the subunit metric (also known as Carnot-Carath\'{e}odory metric, or control distance) can be defined as follows.\par
    For $x,y\in V$ and $\delta>0$, let $C(x,y,\delta)$ denote the collection of absolutely continuous mapping $\varphi:[0,1]\mapsto V$, which satisfying $\varphi(0)=x,\varphi(1)=y$ and
    \[ \varphi'(t)=\sum_{i=1}^{m}a_{i}(t)X_{i}(\varphi(t))\qquad \mbox{a.e for all}~ t\in [0,1] \]
    with $\sum_{k=1}^{m}|a_{k}(t)|^2\leq \delta^2$ for a.e $t\in [0,1]$. From the Chow-Rashevskii theorem (See \cite{MB2014}, Theorem 57) we know that there exist a $\delta>0$ such that $C(x,y,\delta)\neq\varnothing$. Then we can define the subunit metric $d_{X}(x,y)$ as follows
    \[ d_{X}(x,y):=\inf\{\delta>0~|~ \exists \varphi\in C(x,y,\delta)~\mbox{with}~\varphi(0)=x, \varphi(1)=y\}. \]
    Now, we denote
    \[ B_{d_{X}}(x,r):=\{y\in V~|~d_{X}(x,y)<r\} \]
    as the subunit ball induced by the subunit metric $d_{X}(x,y)$. For the volume of the subunit ball,
      a well-known result by Fefferman and Phong \cite{Fefferman} states that for any compact set $K\subset V$, there are constants $c=c(K)>0, \delta_{0}=\delta_{0}(K)>0$ and $\epsilon_{0}>0$ such that for any $x\in K$ and $0<r<\delta_{0}$ we have
    \begin{equation}
\label{2-9}
     B(x,r)\subset B_{d_{X}}(x,cr^{\epsilon_{0}}),
    \end{equation}
where $B(x,r)$ is the ball in the classical Euclidean metric.
Moreover,
 we can precisely estimate the volume of the subunit ball  by Proposition \ref{pro2-6} below.\par

Since $X_{1},\ldots,X_{m}$ together with their commutators of length at most $r_{0}$ span $T_{x}(V)$ at each point $x$ of $V$, we can write the commutators of higher order by means of the following standard notation.\par
 Let $I=(j_{1},\ldots,j_{k})$ ($1\leq j_{i}\leq m$) is a multi-index with length $|I|=k$, $$X_{I}=[X_{j_{1}},[X_{j_{2}},\cdots[X_{j_{k-1}},X_{j_{k}}]\cdots]].$$
The set  $X^{(k)}$ is defined as commutators of length $k$:
\[ X^{(1)}=\{X_{1},\ldots,X_{m}\}, \]
\[ X^{(2)}=\{[X_{1},X_{2}],\ldots,[X_{m-1},X_{m}]\},  \]
\[ X^{(k)}=\{X_{I}|I=(j_{1},\ldots,j_{k}),1\leq j_{i}\leq m, |I|=k \}. \]
Let $Y_{1},\cdots,Y_{q}$ be some enumeration of the components of $X^{(1)},\ldots,X^{(r_{0})}$. If $Y_{i}$ is an element of $X^{(k)}$, we say $Y_{i}$ has formal degree $d(Y_{i})=k$.
By notation in \cite{Stein1985},  for each $n$-tuple of integers $I=(i_{1},\ldots,i_{n}), 1\leq i_{j}\leq q$, we set
\begin{equation}\label{2-10}
\lambda_{I}(x):=\det(Y_{i_{1}},\cdots,Y_{i_{n}})(x).
\end{equation}
(If $Y_{i_{j}}=\sum_{k=1}^{n}a_{jk}(x)\partial_{x_{k}}$, then $\det(Y_{i_{1}},\cdots,Y_{i_{n}})(x)=\det(a_{jk}(x))$). We also set
\[ d(I):=d(Y_{i_{1}})+\cdots+d(Y_{i_{n}}), \]
then we define the $\Lambda(x,r)$ as
\begin{equation}\label{2-11}
  \Lambda(x,r):=\sum_{I}|\lambda_{I}(x)|r^{d(I)},
\end{equation}
where the sum is taken over all $n$-tuples. Now we state the
following proposition obtained by Nagel, Stein and Wainger.
\begin{proposition}[Ball-Box theorem]
\label{pro2-6}
For any compact set $ K \subset V$, there exists $\delta_{0}=\delta_{0}(K)>0$, and $C_{1},C_{2}>0$ such that for all $x\in K$ and $ r\leq \delta_{0}$ we have
\begin{equation}\label{2-12}
  C_{1}\Lambda(x,r)\leq |B_{d_{X}}(x,r)|\leq C_{2}\Lambda(x,r),
\end{equation}
where $|B_{d_{X}}(x,r)|$ is the Lebesgue measure of $B_{d_{X}}(x,r)$.
\end{proposition}
\begin{proof}
See \cite{Morbidelli2000} and \cite{Stein1985}.
\end{proof}

According to \eqref{2-11} and Proposition \ref{pro2-6}, we can deduce that the pointwise homogeneous  dimension of $x$ has the following property.
\begin{equation}\label{2-13}
\nu(x)=\sum_{j=1}^{r_{0}}j(\nu_{j}(x)-\nu_{j-1}(x))=\lim_{r\to 0^{+}}\frac{\log\Lambda(x,r)}{\log r}=\min\{d(I)|\lambda_{I}(x)\neq 0\}.
\end{equation}
Then from the \eqref{2-11},\eqref{2-12} and \eqref{2-13}, we know
that $|B_{d_{X}}(x,r)|$ behaves like $r^{\nu(x)}$  as $r\to
0^{+}$.\par

\section{Explicit estimates of Dirichlet eigenfunctions and Dirichlet eigenvalues}
\subsection{Supremum norm estimates of Dirichlet eigenfunctions}
    The task in this part is to estimate the supremum norm of Dirichlet eigenfunctions for sub-elliptic operator $-\triangle_{X}=\sum_{i=1}^{m}X_{i}^{*}X_{i}$.\par
    For each $i\in \mathbb{N}^{+}$, $\phi_{i}\in H_{X,0}^{1}(\Omega)$ denotes as the $i^{th}$  Dirichlet eigenfunction corresponding with the $i^{th}$  Dirichlet eigenvalue $\lambda_{i}$, we have $(\triangle_{X}+\lambda_{i})\phi_{i}=0$. According to the regularity results of Derridj in \cite{Derridj}, we know that $\phi_{i}\in C^{\infty}(\overline{\Omega})$ and $\phi_{i}|_{\partial\Omega}=0$. Moreover, the sequence of eigenfunctions $\{\phi_{i}\}_{i=1}^{\infty}$ constitutes an orthogonal basis in $H_{X,0}^{1}(\Omega)$ with $\|\phi_{i}\|_{L^2(\Omega)}=1$, which is also a standard orthogonal basis in $L^2(\Omega)$. Furthermore, we have the following estimates of $L^{\infty}$-norm for the Dirichlet eigenfunction $\phi_{i}$.
\begin{proposition}\label{pro3-1}
Suppose that  $X=(X_{1},X_{2},\cdots,X_{m})$ satisfy the conditions of Theorem \ref{thm1}. We have
\begin{equation}\label{3-1}
\|\phi_{i}\|_{\infty}\leq C_{1}\cdot\lambda_{i}^{\frac{\tilde{\nu}}{4}},
\end{equation}
where $C_{1}$ is a positive constant depending on $X$ and $\Omega$, $\tilde{\nu}$ is the generalized M\'{e}tivier index on $\Omega$, $\|\cdot\|_{\infty}$ denotes the $L^{\infty}$-norm on $\Omega$.
\end{proposition}
\begin{proof}
Since $-\triangle_{X}\phi_{i}=\lambda_{i}\phi_{i}$, then
\begin{equation}\label{3-2}
\int_{\Omega}X\phi_{i}\cdot Xudx=\lambda_{i}\int_{\Omega}\phi_{i}u dx,\qquad \forall u\in H_{X,0}^{1}(\Omega).
\end{equation}
For any constant $s\geq 2$, we take $u=|\phi_{i}|^{s-1}\cdot \text{sgn}(\phi_{i})$. Since
\[ Xu=(s-1)|\phi_{i}|^{s-2}X\phi_{i},\]
we can deduce that $u\in H_{X,0}^{1}(\Omega)$. Therefore \eqref{3-2} implies that
\begin{equation}\label{3-3}
\lambda_{i}\int_{\Omega}|\phi_{i}|^{s}dx=(s-1)\int_{\Omega}|\phi_{i}|^{s-2}|X\phi_{i}|^2dx.
\end{equation}
For each $ f\in H_{X,0}^{1}(\Omega)$, we know that $X|f|=\text{sgn}(f)Xf$ (cf. \cite{Garofalo1996} Lemma 3.5). Moreover, \eqref{3-3} gives
\begin{equation}\label{3-4}
  \lambda_{i}\int_{\Omega}|\phi_{i}|^{s}dx=(s-1)\int_{\Omega}|\phi_{i}|^{s-2}|X\phi_{i}|^2dx=(s-1)\int_{\Omega}|\phi_{i}|^{s-2}|X|\phi_{i}||^2dx.
\end{equation}

\par
On the other hand, for any non-negative function $v\in H_{X,0}^{1}(\Omega)\cap L^{\infty}(\Omega)$ and
 any constant $s\geq 2$, integrating by parts and applying the weighted Sobolev inequality  \eqref{2-4} we
 have
\begin{align}\label{3-5}
-(s-1)\int_{\Omega}v^{s-2}|Xv|^2dx&= -\frac{4(s-1)}{s^2}\int_{\Omega}|X(v^{\frac{s}{2}})|^2dx \notag \\
   &\leq -\frac{4C(s-1)}{s^2}\left(\int_{\Omega}|v|^{\frac{s\tilde{\nu}}{\tilde{\nu}-2}}dx \right)^{\frac{\tilde{\nu}-2}{\tilde{\nu}}} \notag \\
   & \leq -\frac{2C}{s}\left(\int_{\Omega}|v|^{\frac{s\tilde{\nu}}{\tilde{\nu}-2}}dx \right)^{\frac{\tilde{\nu}-2}{\tilde{\nu}}},
\end{align}
where $C$ is the Sobolev constant in \eqref{2-4}. Thus if $v=|\phi_{i}|$, then $v\in H_{X,0}^{1}(\Omega)\cap L^{\infty}(\Omega)$.
Hence \eqref{3-4} and \eqref{3-5} assert that
\[ \int_{\Omega}|\phi_{i}|^{s}dx\geq \frac{2C}{s\lambda_{i}}\left(\int_{\Omega}|\phi_{i}|^{\frac{s\tilde{\nu}}{\tilde{\nu}-2}} dx\right)^{\frac{\tilde{\nu}-2}{\tilde{\nu}}},\]
which can be rewritten as
\[ \left(\frac{2C}{s\lambda_{i}}\right)^{\frac{1}{s}}\|\phi_{i}\|_{s\beta}\leq \|\phi_{i}\|_{s} \]
for all $s\geq 2$, with $\beta=\frac{\tilde{\nu}}{\tilde{\nu}-2}$.
Here $\|\phi_{i}\|_{s}$ is the $L^{s}$-norm of $\phi_{i}$. Setting
$s=2\beta^{j}\geq 2$, respectively for $j=0,1,2,\cdots$, then we
have
\[ \|\phi_{i}\|_{2\beta^{j+1}}\leq \left(\frac{\beta^{j}\lambda_{i}}{C}\right)^{\frac{1}{2\beta^{j}}}\|\phi_{i}\|_{2\beta^{j}}. \]
Iterating this estimate and using $\|\phi_{i}\|_{2}=1$, we conclude that
\begin{align*}
\|\phi_{i}\|_{2\beta^{j+1}}&\leq \prod_{k=0}^{j}\left(\frac{\beta^{k}\lambda_{i}}{C}\right)^{\frac{1}{2\beta^{k}}}\\
&=\beta^{\sum_{k=0}^{j}\frac{k}{2\beta^{k}}}\cdot C^{-\frac{1}{2}\sum_{k=0}^{j}\beta^{-k}}\cdot \lambda_{i}^{\frac{1}{2}\sum_{k=0}^{j}\beta^{-k}}\\
&=\beta^{\frac{1}{2}\cdot\left(\frac{\beta^{-1}(1-\beta^{-j})}{(1-\beta^{-1})^{2}}-\frac{j}
{(1-\beta^{-1})\beta^{j+1}} \right)}\cdot C^{-\frac{1}{2}\cdot\frac{1-\beta^{-(j+1)}}{1-\beta^{-1}}}\cdot\lambda_{i}^{\frac{1}{2}
\cdot
\frac{1-\beta^{-(j+1)}}{1-\beta^{-1}}}.
\end{align*}

Letting $j\to \infty$ and applying the fact that $\lim_{p\to\infty}\|\phi_{i}\|_{p}=\|\phi_{i}\|_{\infty}$,
we obtain
\[ \|\phi_{i}\|_{\infty}\leq \beta^{\frac{\beta}{2(\beta-1)^2}}\cdot C^{-\frac{\beta}{2(\beta-1)}}\cdot\lambda_{i}^{\frac{\beta}{2(\beta-1)}}=C_{1}
\cdot\lambda_{i}^{\frac{\tilde\nu}{4}}, \]
where $C_{1}$ is a positive constant depends on $C$ and $\tilde\nu$.
\end{proof}

\subsection{An explicit lower bound of Dirichlet eigenvalues}
 The aim in this part is to get an explicit lower bound of the sub-elliptic Dirichlet eigenvalue $\lambda_{k}$. Although the lower bound of $\lambda_{k}$ may not be precise, it is useful in the process of estimating Dirichlet heat kernel of $\triangle_{X}$.
\begin{proposition}\label{Pro3-2}
Suppose $X=(X_{1},X_{2},\cdots, X_{m})$ satisfy the conditions of
Theorem \ref{thm1}. Then we have
\begin{equation}\label{3-6}
  \lambda_{k}\geq C\cdot k^{\frac{2\epsilon_{0}}{n}},~~\forall k\geq 1,
\end{equation}
where the positive constant $C$ depends on vector fields $X$ and
$\Omega$, and $\epsilon_{0}$ is a positive constant in Proposition
\ref{subelliptic-estimate}.
\end{proposition}
Our proof of Proposition \ref{Pro3-2} is inspired by Chen and Luo's approach in \cite{chenluo} and the work of Li and Yau in \cite{Yau1983}. We need several lemmas to prove Proposition \ref{Pro3-2}.

\begin{lemma}
\label{Chen-luo1}
Let $\{\phi_{j}\}_{j=1}^{k}$ be the set of orthonormal eigenfunctions corresponding to the Dirichlet eigenvalues $\{\lambda_{j}\}_{j=1}^{k}$. Define
\[ \Phi(x,y):=\sum_{j=1}^{k}\phi_{j}(x)\phi_{j}(y). \]
Then we have
\[ \int_{\Omega}\int_{\mathbb{R}^n}|\hat{\Phi}(z,y)|^2dzdy =k,~~\mbox{and}~~ \int_{\Omega}|\hat{\Phi}(z,y)|^2dy\leq (2\pi)^{-n}|\Omega|, \]
where $\hat{\Phi}(z,y)$ is the partial Fourier transformation of $\Phi(x,y)$ in the $x$-variable
\[ \hat{\Phi}(z,y)=(2\pi)^{-\frac{n}{2}}\int_{\mathbb{R}^n}\Phi(x,y)e^{-ix\cdot z}dx. \]

\end{lemma}
\begin{proof}
See Lemma 3.1 in \cite{chenluo}.
\end{proof}
\begin{lemma}
\label{Chen-luo2}
Let $f$ be a real-valued function defined on $\mathbb{R}^n$ with $0\leq f\leq M_{1}$. If
\[ \int_{\mathbb{R}^n}|z|^{2\epsilon_{0}}f(z)dz\leq M_{2}, \]
with $\epsilon_{0}>0$, then we have the following inequality,
\[ \int_{\mathbb{R}^n}f(z)dz\leq \left(\frac{n+2\epsilon_{0}}{n}\right)^{\frac{n}{n+2\epsilon_{0}}}(M_{1}B_{n})^{\frac{2\epsilon_{0}}{n+2\epsilon_{0}}}M_{2}^{\frac{n}{n+2\epsilon_{0}}}, \]
where $B_{n}$ is the volume of the unit ball in $\mathbb{R}^n$.
\end{lemma}
\begin{proof}
First, we can choose $R$ such that
\[ \int_{\mathbb{R}^n}|z|^{2\epsilon_{0}}g(z)dz=M_{2},\]
where
\[ g(z)=\left\{
          \begin{array}{ll}
            M_{1}, & \hbox{$|z|<R$;} \\[2mm]
            0, & \hbox{$|z|\geq R$.}
          \end{array}
        \right.\]
Then $(|z|^{2\epsilon_{0}}-R^{2\epsilon_{0}})(f(z)-g(z))\geq 0$. Hence we get
\begin{equation}\label{3-7}
  R^{2\epsilon_{0}}\int_{\mathbb{R}^n}(f(z)-g(z))dz\leq \int_{\mathbb{R}^n}|z|^{2\epsilon_{0}}(f(z)-g(z))dz\leq 0.
\end{equation}
Note that
 \begin{equation}\label{3-8}
 M_{2}=\int_{\mathbb{R}^n}|z|^{2\epsilon_{0}}g(z)dz=M_{1}\int_{0}^{R}r^{n-1+2\epsilon_{0}}\omega_{n-1}dr=\frac{M_{1}\omega_{n-1}R^{n+2\epsilon_{0}}}{n+2\epsilon_{0}},
\end{equation}
where $\omega_{n-1}$ is the area of the unit sphere in $\mathbb{R}^n$. By the definition of $g(z)$, we know
 \begin{equation}\label{3-9}
  \int_{\mathbb{R}^n}g(z)dz=M_{1}B_{n}R^{n},
\end{equation}
 where $B_{n}$ is the volume of the unit ball in $\mathbb{R}^n$.\par
Since $nB_{n}=\omega_{n-1}$, then \eqref{3-7},\eqref{3-8} and \eqref{3-9} give
\[ \int_{\mathbb{R}^n}f(z)dz\leq \int_{\mathbb{R}^n}g(z)dz\leq\left(\frac{n+2\epsilon_{0}}{n}\right)^{\frac{n}{n+2\epsilon_{0}}}(M_{1}B_{n})^{\frac{2\epsilon_{0}}{n+2\epsilon_{0}}}M_{2}^{\frac{n}{n+2\epsilon_{0}}}. \]
\end{proof}

Now, we can prove Proposition \ref{Pro3-2}.

\begin{proof}[Proof of Proposition \ref{Pro3-2} ]
For $\Phi(x,y)=\sum_{j=1}^{k}\phi_{j}(x)\phi_{j}(y)$, we know that $\Phi(x,y)\in H_{X,0}^{1}(\Omega)$ with respect to $x$. By Proposition \ref{subelliptic-estimate} we can deduce that
\begin{equation}\label{3-10}
  \left\||\nabla|^{\epsilon_{0}}u\right\|_{L^2(\Omega)}^2\leq C\|Xu\|_{L^2(\Omega)}^2~~\mbox{for all}~~u\in H_{X,0}^{1}(\Omega),
\end{equation}
where $\nabla=(\partial_{x_{1}},\ldots,\partial_{x_{n}})$, $|\nabla|^{\epsilon_{0}}$ is a pseudo-differential operator with the symbol $|\xi|^{\epsilon_{0}}$,  $C>0$ is a constant depends on $X$ and $\Omega$, and $\epsilon_{0}$ is a positive constant in Proposition \ref{subelliptic-estimate}.
Then, by using Placherel's formula, we have
\begin{equation}
\label{3-11}
\int_{\mathbb{R}^n}\int_{\Omega}|z|^{2\epsilon_{0}}\left|\hat{\Phi}(z,y)\right|^2dydz=\int_{\mathbb{R}^n}\int_{\Omega}\left||\nabla|^{\epsilon_{0}}\Phi(x,y)  \right|^2dydx=\int_{\Omega}\int_{\Omega}\left||\nabla|^{\epsilon_{0}}\Phi(x,y)  \right|^2dydx.
\end{equation}
Combining \eqref{3-10} and \eqref{3-11}, we get
\begin{equation}
\label{3-12}
\int_{\Omega}\int_{\Omega}\left||\nabla|^{\epsilon_{0}}\Phi(x,y)  \right|^2dydx\leq C\int_{\Omega}\int_{\Omega}|X_{x}\Phi(x,y)|^2dxdy.
 \end{equation}
On the other hand, we can deduce that
\begin{align}
\label{3-13}
\int_{\Omega}\int_{\Omega}|X_{x}\Phi(x,y)|^2dxdy&=\int_{\Omega}\left(\sum_{i=1}^{m}\int_{\Omega}\left|\sum_{j=1}^{k}(X_{i}\phi_{j}(x))\phi_{j}(y)  \right|^2dx\right)dy \notag\\
&=\sum_{i=1}^{m}\left(\int_{\Omega}\sum_{j=1}^{k}|X_{i}\phi_{j}(x)|^2 \right)dx\notag\\
&=\sum_{j=1}^{k}(X\phi_{j},X\phi_{j})_{L^2(\Omega)}=\sum_{j=1}^{k}(-\triangle_{X}\phi_{j},\phi_{j})_{L^2(\Omega)}\notag\\
&=\sum_{j=1}^{k}\lambda_{j}.
\end{align}
It follows from estimates \eqref{3-11}, \eqref{3-12} and \eqref{3-13} that
\[ \int_{\mathbb{R}^n}\int_{\Omega}|z|^{2\epsilon_{0}}\left|\hat{\Phi}(z,y)\right|^2dydz\leq C\sum_{j=1}^{k}\lambda_{j}. \]
Now we take
\[ f(z)=\int_{\Omega}\left|\hat{\Phi}(z,y)\right|^2dy,\quad M_{1}=(2\pi)^{-n}|\Omega|,\quad M_{2}=C\sum_{j=1}^{k}\lambda_{j}. \]
Then, due to Lemma \ref{Chen-luo1} and Lemma \ref{Chen-luo2}, we have
\[ k\leq \left(\frac{n+2\epsilon_{0}}{n}\right)^{\frac{n}{n+2\epsilon_{0}}}[(2\pi)^{-n}|\Omega|B_{n}]^{\frac{2\epsilon_{0}}{n+2\epsilon_{0}}}\left(C\sum_{j=1}^{k}\lambda_{j}\right)^{\frac{n}{n+2\epsilon_{0}}}. \]
Consequently,
\[ \sum_{j=1}^{k}\lambda_{j}\geq C\cdot k^{1+\frac{2\epsilon_{0}}{n}}. \]
Therefore, by $\lambda_{i}\leq \lambda_{i+1}$ we have
\[ \lambda_{k}\geq C\cdot k^{\frac{2\epsilon_{0}}{n}}~~\mbox{for all}~~k\geq 1. \]
\end{proof}
\section{Sub-elliptic Dirichlet heat kernel}
    We construct the sub-elliptic Dirichlet heat kernel of $\triangle_{X}$ in this section. Our approach is
     similar to that in Li's work \cite{Li2012} in the classical case.
The sub-elliptic Dirichlet heat kernel $h_D(x,y,t)$ of $\triangle_{X}$ is the fundamental solution of the degenerate heat operator $\partial_{t}-\triangle_{X}$. That is, for any fixed point $y\in \Omega$, $h_D(x,y,t)$ is the solution of
\begin{equation}\label{4-1}
\left(\frac{\partial}{\partial t}-\triangle_{X}\right)h_D(x,y,t)=0 \quad \mbox{for all}~~ (x,t)\in\Omega\times (0,+\infty),
\end{equation}
and satisfies following properties
\begin{equation}\label{4-2}
h_{D}(x,y,t)\in C^{\infty}(\Omega\times\Omega\times\mathbb{R}^{+})\cap C(\overline\Omega\times\overline\Omega\times\mathbb{R}^{+})~~\mbox{and}~~ h_{D}(x,y,t)\in H_{X,0}^{1}(\Omega).
\end{equation}

\begin{equation}\label{4-3}
\lim_{t\to 0^{+}}\int_{\Omega}h_D(x,y,t)\varphi(y)dy=\varphi(x),~~\mbox{for all } \varphi\in C_{0}^{\infty}(\Omega).
\end{equation}
\begin{equation}\label{4-4}
 h_D(x,y,t)=0 \mbox{ when } x,~y\in\partial\Omega,\qquad h_{D}(x,y,t)=h_{D}(y,x,t).
\end{equation}

 \begin{equation}\label{4-5}
h_D(x,y,t+s)=\int_{\Omega}h_D(x,z,t)h_D(z,y,s)dz,~~ \mbox{for all } s, t>0.
\end{equation}

\begin{equation}\label{4-6}
\begin{aligned}
&h_D(x,y,t)>0,~~\mbox{for all } (x,y,t)\in \Omega\times\Omega\times(0,+\infty),\\
&\int_{\Omega}h_D(x,y,t)dy\leq 1,~~ \mbox{for all } (x,t)\in\Omega\times(0,+\infty).
\end{aligned}
\end{equation}

\par
Since the Dirichlet heat kernel $h_D(x,y,t)$ is the fundamental solution of $\frac{\partial}{\partial t}-\triangle_{X}$. Thus for a function $f_0(x)\in L^2(\Omega)$,
the function
\begin{equation}\label{4-7}
f(x,t)=\int_{\Omega}h_D(x,y,t)f_{0}(y)dy
\end{equation}
will solve the degenerate heat equation
\begin{equation}\label{4-8}
\left(\triangle_{X}-\frac{\partial}{\partial t}\right)f(x,t)=0, \mbox{ for }(x,t)\in \Omega\times(0,+\infty),
\end{equation}
and satisfies
\begin{equation}\label{4-9}
\lim_{t\to 0+}f(x,t)=f_{0}(x)~~ \mbox{in}~~ L^2(\Omega),~~ \mbox{ and }~~ f(x,t)=0 \mbox{ on }\partial\Omega\times (0,+\infty).
\end{equation}
\par
Recall that the sequence of eigenfunctions
$\{\phi_{i}\}_{i=1}^{\infty}$ is  a standard orthogonal basis in
$L^2(\Omega)$, that implies that a function $f_{0}\in L^2(\Omega)$
can be written in the form
\[ f_{0}=\sum_{k=1}^{\infty}a_{k}\phi_{k}(x) \mbox{ with } a_{k}=\int_{\Omega}f_{0}\phi_{k}dx. \]
Formally, the function $f(x,t)$ can be given by
\begin{equation}\label{4-10}
f(x,t)=\sum_{i=1}^{\infty}e^{-\lambda_{i}t}a_{i}\phi_{i}(x),
\end{equation}
which satisfies the \eqref{4-8} with initial-boundary condition \eqref{4-9}. Comparing \eqref{4-7} and \eqref{4-10}, we can deduce that the Dirichlet heat kernel of $\triangle_{X}$ on $\Omega$ can be defined as
\[ h_D(x,y,t)=\sum_{k=1}^{\infty}e^{-\lambda_{k}t}\phi_{k}(x)\phi_{k}(y). \]

    In fact, we have the following proposition.\par
\begin{proposition}\label{prop4-1}
Let  $X=(X_{1},\cdots,X_{m})$ with conditions the same as Theorem \ref{thm1}. Then the sub-elliptic operator $\triangle_{X}$ has a
 Dirichlet heat kernel $h_{D}(x,y,t)$ which is well defined on $\overline{\Omega}\times\overline{\Omega}\times(0,+\infty)$ by
\begin{equation}\label{4-11}
h_{D}(x,y,t)=\sum_{i=1}^{\infty}e^{-\lambda_{i}t}\phi_{i}(x)\phi_{i}(y).
\end{equation}
Furthermore, $h_{D}(x,y,t)$ is uniquely determined and satisfies
properties \eqref{4-1} to \eqref{4-6}.
\end{proposition}
\begin{proof}
 We begin by establishing the uniform convergence of the series \eqref{4-11}.
 By Proposition \ref{pro3-1} and recall that $\phi_{i}\in C^{\infty}(\overline{\Omega})$, we have, for $t>0$,
\begin{equation}\label{4-12}
 e^{-\lambda_{i}t}|\phi_{i}(x)|\cdot|\phi_{i}(y)|\leq e^{-\lambda_{i}t}\|\phi_{i}\|_{\infty}^2\leq C_{1}^2e^{-\lambda_{i}t}\lambda_{i}^{\frac{\tilde{\nu}}{2}}.
\end{equation}
Now, we use the inequality (cf. \cite{Chavel1984} Chapter VII)
\begin{equation}\label{4-13}
  e^{-z}z^{\alpha}\leq \alpha^{\alpha}e^{-\alpha} \qquad \mbox{for all }z>0,~\alpha>0.
\end{equation}
Putting $z=\frac{1}{2}\lambda_{i}t$, $\alpha=\frac{\tilde{\nu}}{2}$ into \eqref{4-13}, we get
\begin{equation}\label{4-14}
e^{-\lambda_{i}t}\lambda_{i}^{\frac{\tilde{\nu}}{2}}\leq \tilde{\nu}^{\frac{\tilde{\nu}}{2}}e^{-\frac{\tilde{\nu}}{2}}t^{-\frac{\tilde{\nu}}{2}}
e^{-\frac{1}{2}\lambda_{i}t}.
\end{equation}
Hence, \eqref{4-12} and \eqref{4-14} imply that
\begin{equation}\label{4-15}
 \sum_{i=1}^{\infty}e^{-\lambda_{i}t}|\phi_{i}(x)|\cdot|\phi_{i}(y)|\leq C_{1}^2\cdot\tilde{\nu}^{\frac{\tilde{\nu}}{2}}e^{-\frac{\tilde{\nu}}{2}}
t^{-\frac{\tilde{\nu}}{2}}\sum_{k=1}^{\infty}e^{-\frac{1}{2}\lambda_{k}t},
\end{equation}
where $C_{1}>0$ is a constant depending on $C$ and $\tilde\nu$. The explicit lower bound of Dirichlet eigenvalue $\lambda_{k}$ which is established in Proposition \ref{Pro3-2} allows us to obtain that
\begin{equation}\label{4-16}
\sum_{i=1}^{\infty}e^{-\lambda_{i}t}|\phi_{i}(x)|\cdot|\phi_{i}(y)|\leq C_{1}^2\cdot \tilde{\nu}^{\frac{\tilde{\nu}}{2}}e^{-\frac{\tilde{\nu}}{2}}t^{-\frac{\tilde{\nu}}{2}}
\sum_{k=1}^{\infty}e^{-\frac{C}{2}k^{\frac{2\epsilon_{0}}{n}}t},
\end{equation}
where $\epsilon_{0}$ and $C$ are positive constants in Proposition \ref{Pro3-2}. The estimate \eqref{4-16} implies the series \eqref{4-11} uniformly convergent on $\overline{\Omega}\times\overline{\Omega}\times[a,\infty)$ for any $a>0$. Thus the sub-elliptic Dirichlet heat kernel $h_D(x,y,t)=\sum_{i=1}^{\infty}e^{-\lambda_{i}t}\phi_{i}(x)\phi_{i}(y)$
is well-defined. Moreover, it can be clearly seen that $h_{D}(x,y,t)$
 satisfies \eqref{4-4}.\par
We denote $Sh_N(x,y,t)$ as the sum of the first $N$ terms of the series \eqref{4-11}, i.e.
\begin{equation}\label{4-17}
Sh_{N}(x,y,t):=\sum_{i=1}^{N}e^{-\lambda_{i}t}\phi_{i}(x)\phi_{i}(y).
\end{equation}
Since
\[ \int_{\Omega}(X\phi_{i}(x))\cdot(X\phi_{j}(x))dx=\delta_{ij}\lambda_{i}=
\left\{
\begin{array}{ll}
\lambda_{i}, & \hbox{$i=j$;} \\[3mm]
0, & \hbox{$i\neq j$.}
\end{array}
\right.
\]
Similarly, for any fixed $t>0$, we have
\begin{align*}
&\int_{\Omega}\left|X_{y}(Sh_{N+k}(x,y,t)-Sh_{N}(x,y,t))\right|^2dy\\
&\int_{\Omega}\left|X_{y}\left(\sum_{i=N+1}^{N+k}e^{-\lambda_{i}t}\phi_{i}(x)
\phi_{i}(y)\right) \right|^2dy\\
&=\int_{\Omega}\left|\sum_{i=N+1}^{N+k}e^{-\lambda_{i}t}\phi_{i}(x)X_{y}\phi_{i}(y) \right|^2dy\\
&=\int_{\Omega}\sum_{i,j=N+1}^{N+k}e^{-(\lambda_{i}+\lambda_{j})t}\phi_{i}(x)\phi_{j}(x)
[X_{y}\phi_{i}(y)\cdot X_{y}\phi_{j}(y)]dy\\
&=\sum_{i=N+1}^{N+k}e^{-2\lambda_{i}t}\lambda_{i}\phi_{i}^2(x)\to 0\quad (\mbox{for any }k\in \mathbb{N}^{+},\mbox{ as }N\to+\infty).
\end{align*}
Thus, it gives us that $Sh_{N}(x,y,t)\to h_{D}(x,y,t)$ uniformly as $N\to+\infty$ in $H_{X,0}^{1}(\Omega)$ for $t>0$. Consequently, for any fixed $(y,t)\in \Omega\times (0,+\infty)$, $h_{D}(x,y,t)\in H_{X,0}^{1}(\Omega)$ with respect to $x$.\par

 Furthermore, for a fixed point $y\in \Omega$  and $N\in\mathbb{Z}^{+}$, $u_{y,N}(x,t)=Sh_{N}(x,y,t)$ is a solution of the degenerate heat equation \eqref{4-1}. The uniform convergence of $u_{y,N}(x,t)$  implies that $h_{D}(x,y,t)$ is a weak solution of \eqref{4-1} with respect to $(x,t)$. Analogously, it is easy to verify that $h_{D}(x,y,t)$ is also a weak solution of equation $[\partial_{t}-\frac{1}{2}(\triangle_{X}^{x}+\triangle_{X}^{y})]u(x,y,t)=0$, since for each $N$, $Sh_{N}(x,y,t)$ is a solution of $[\partial_{t}-\frac{1}{2}(\triangle_{X}^{x}+\triangle_{X}^{y})]u(x,y,t)=0$. Then the hypo-ellipticity of $\partial_{t}-\frac{1}{2}(\triangle_{X}^{x}+\triangle_{X}^{y})$ implies that $h_{D}(x,y,t)\in C^{\infty}(\Omega\times\Omega\times(0,+\infty))$. Also, the uniform convergence of $Sh_{N}(x,y,t)$ on $\overline\Omega\times\overline\Omega\times [a,+\infty)$ for any $a>0$ gives $h_{D}(x,y,t)\in C(\overline\Omega\times\overline\Omega\times(0,+\infty))$.\par

Now recall that the sequence of Dirichlet eigenfunctions $\{\phi_{k}\}_{k=1}^{\infty}$ constitutes a standard orthogonal basis in $L^2(\Omega)$. Given a function $f_{0}(x)\in L^2(\Omega)$, we have
\[ f_{0}(x)=\sum_{i=1}^{\infty}a_{i}\phi_{i}(x)~~\mbox{in}~~ L^2(\Omega), \]
where $a_{i}=\int_{\Omega}f_{0}(y)\phi_{i}(y)dy $.
In terms  of Parseval's identity we know
\begin{equation}\label{4-18}
    \sum_{i=1}^{\infty}a_{i}^2= \|f_{0}\|_{L^2(\Omega)}^{2}<+\infty.
\end{equation}
Furthermore, for any $t>0$, we have
\begin{align*}
f(x,t)&=\int_{\Omega}h_{D}(x,y,t)f_{0}(y)dy\\
&=\int_{\Omega}\left(\sum_{i=1}^{\infty}e^{-\lambda_{i}t}\phi_{i}(x)\phi_{i}(y)\right)
\left(\sum_{j=1}^{\infty}a_{j}\phi_{j}(y)\right)dy\\
&=\sum_{i=1}^{\infty}e^{-\lambda_{i}t}a_{i}\phi_{i}(x)~~\mbox{in}~~ L^2(\Omega).
\end{align*}
Since
\[ \sum_{i=1}^{\infty}e^{-\lambda_{i}t}|a_{i}|\cdot|\phi_{i}(x)|\leq \|f_{0}\|_{L^2(\Omega)}\cdot\sum_{i=1}^{\infty}e^{-\lambda_{i}t}\cdot|\phi_{i}(x)|\leq C_{1}\|f_{0}\|_{L^2(\Omega)}\cdot\sum_{i=1}^{\infty}e^{-\lambda_{i}t}
\lambda_{i}^{\frac{\tilde\nu}{4}}. \]
Then by using similar approach as above, we know that $\sum_{i=1}^{\infty}e^{-\lambda_{i}t}a_{i}\phi_{i}(x)$ converges uniformly on $\overline\Omega\times [a,+\infty)$ for any $a>0$, which implies $f(x,t)$ is a weak solution of the degenerate heat equation \eqref{4-8} and agrees with the Dirichlet boundary condition in \eqref{4-9}. Moreover, the hypo-ellipticity of $\partial_{t}-\triangle_{X}$ tells us $f(x,t)\in C^{\infty}(\Omega\times(0,+\infty))\cap C(\overline\Omega\times(0,+\infty))$. \par
 In order to verify that $f(x,t)$ satisfies the initial condition in \eqref{4-9}, it suffices to prove that $f(x,t)=\int_{\Omega}h_{D}(x,y,t)f_{0}(y)dy\to f_{0}(x)$ as $t\to 0^{+}$ in $L^2(\Omega)$. It derives, in fact, that
\begin{align*}
\|f(x,t)-f_{0}(x)\|_{L^2(\Omega)}^{2}&=\left\|\sum_{i=1}^{\infty}(e^{-\lambda_{i}t}-1)
a_{i}\phi_{i}(x) \right\|_{L^2(\Omega)}^{2}\\
&=\sum_{i=1}^{\infty}a_{i}^2(1-e^{-\lambda_{i}t})^2.
\end{align*}
Thus  identity \eqref{4-18} implies that $\sum_{i=1}^{\infty}a_{i}^2(1-e^{-\lambda_{i}t})^2$ converges uniformly on $t\in[0,+\infty)$. Therefore, we obtain
\begin{equation*}
\lim_{t\to 0^{+}}\|f(x,t)-f_{0}(x)\|_{L^2(\Omega)}^{2}=\lim_{t\to 0^{+}}\sum_{i=1}^{\infty}a_{i}^2(1-e^{-\lambda_{i}t})^2=0,
\end{equation*}
which means that $f(x,t)$ allows the initial condition in \eqref{4-9}.\par
If we take $u(x,t)=\int_{\Omega}h_{D}(x,y,t)\varphi(y)dy-\varphi(x)$ for any  $\varphi\in C_{0}^{\infty}(\Omega)$,
 then for any $t>0$, we have $u(x,t)\in L^2(\Omega)\cap C^{\infty}(\Omega)$. Moreover, the symmetry of $h_{D}(x,y,t)$ in $x$ and $y$ gives
\begin{align*}
(\triangle_{X})^{k}u(x,t)&=\int_{\Omega}(\triangle_{X}^{x})^{k}h_{D}(x,y,t)\varphi(y)dy
-(\triangle_{X}^{x})^{k}\varphi(x)\\
&=\int_{\Omega}(\triangle_{X}^{y})^{k}h_{D}(x,y,t)\varphi(y)dy-(\triangle_{X}^{x})^{k}
\varphi(x)\\
&=\int_{\Omega}h_{D}(x,y,t)(\triangle_{X}^{y})^{k}\varphi(y)dy-(\triangle_{X}^{x})^{k}
\varphi(x).
\end{align*}
Thus, we know that $(\triangle_{X})^{k}u(x,t)\in L^2(\Omega)$ for $t>0$, and $k\in\mathbb{N}^{+} $. Meanwhile, by Corollary \ref{corollary2-1}, we have for any $\xi(x)\in C_{0}^{\infty}(\Omega)$,
\begin{align}\label{4-19}
\sup_{x\in\Omega}|\xi(x)u(x,t)|&\leq C\sum_{k=0}^{N}\|(\triangle_{X})^{k}u(x,t)\|_{L^2(\Omega)} \notag\\
&=C\sum_{k=0}^{N}\left\|\int_{\Omega}h_{D}(x,y,t)
(\triangle_{X}^{y})^{k}\varphi(y)dy-(\triangle_{X}^{x})^{k}\varphi(x)\right\|_{L^2(\Omega)}.
\end{align}
Hence from \eqref{4-7} and \eqref{4-9}, the estimate \eqref{4-19} shows that for any cut-off function $\xi(x)\in C_{0}^{\infty}(\Omega)$ we have
\begin{equation}\label{4-20}
\lim_{t\to 0^{+}}\sup_{x\in\Omega}\left|\xi(x)u(x,t)\right|=\lim_{t\to 0^{+}}\sup_{x\in\Omega}\left|\xi(x)\cdot\left(\int_{\Omega}h_{D}(x,y,t)\varphi(y)dy-
\varphi(x)\right) \right|=0.
\end{equation}
Since the cut-off function $\xi(x)$ is arbitrary, then for any given $x\in\Omega$, \eqref{4-20} gives that
\[\lim_{t\to 0^{+}}\int_{\Omega}h_{D}(x,y,t)\varphi(y)dy=\varphi(x).\]
This completes the proof of \eqref{4-3}.\par
Also, we have for $t>0$ and $s>0$,
\begin{align*}
\int_{\Omega}h_D(x,z,t)h_D(z,y,s)dz&=\int_{\Omega}\left(\sum_{i=1}^{\infty}
e^{-\lambda_{i}t}\phi_{i}(x)\phi_{i}(z) \right)\left(\sum_{j=1}^{\infty}e^{-\lambda_{j}s}\phi_{j}(z)\phi_{j}(y) \right)dz\\
 &=\sum_{i=1}^{\infty}e^{-\lambda_{i}(t+s)}\phi_{i}(x)\phi_{i}(y)\\
 &=h_D(x,y,t+s),
\end{align*}
which yields to \eqref{4-5}.\par
 Finally, we only need to verify \eqref{4-6} and the uniqueness of $h_{D}(x,y,t)$.\par
We firstly show that $h_{D}(x,y,t)\geq 0$. Actually, if there exists $(x_{0},y_{0},t_{0})\in (\Omega\times\Omega\times (0,+\infty))$ in which $ h_{D}(x_{0},y_{0},t_{0})<0$, then there exist $0<\delta<t_{0}$ and $\alpha>0$, such that $B(x_{0},\delta)\subset \Omega$, $B(y_{0},\delta)\subset \Omega$ and
 for each $(x,y,t)\in B(x_{0},\delta)\times B(y_{0},\delta)\times(t_0-\delta, t_0+\delta)$, we have
\[ h_{D}(x,y,t)<-\alpha<0. \]
Thus, we can find a function $f_{0}\in C_{0}^{\infty}(B(y_{0},\delta))$ with $0\leq f_{0}\leq 1$, such that
\[ f_{0}(y)=\left\{
              \begin{array}{ll}
                1, & \hbox{$y\in B(y_{0},\frac{\delta}{2})$ ;} \\[3mm]
                0, & \hbox{$y\in \Omega\setminus B(y_{0},\delta)$.}
              \end{array}
            \right.\]
Then
\[ f(x,t)=\int_{\Omega}h_{D}(x,y,t)f_{0}(y)dy =\int_{B(y_{0},\delta)}h_{D}(x,y,t)f_{0}(y)dy. \]
In particular, we have
\begin{equation}\label{4-21}
f(x_{0},t_{0})=\int_{B(y_{0},\delta)}h_{D}(x_{0},y,t_{0})f_{0}(y)dy<0.
\end{equation}
Given $U_{T}=\Omega\times(0,T]$ for some $T> t_{0}$. From above arguments, we can conclude that $f(x,t)\in C(\overline{U_{T}})\cap C^{\infty}(U_{T})$. Since $f(x,t)\geq 0$ in the parabolic boundary $\partial_{p}U_{T}=(\Omega\times\{0\})\cup(\partial\Omega\times[0,T])$, it implies that $f(x,t)\geq 0$ in $\Omega\times(0,T)$ according to the weak maximum principle for the degenerate parabolic equation (cf. Proposition 2.2 in  \cite{Lanconelli}, also see Proposition 3.6 in \cite{Brandolini}).
This is a contradiction with \eqref{4-21}. Hence, we obtain $h_{D}(x,y,t)\geq 0$ for $(x,y,t)\in \Omega\times\Omega\times(0,+\infty)$.\par
Secondly, we assume that  $h_{D}(x',y',t')=0$ for some $ (x',y',t')\in \Omega\times\Omega\times(0,+\infty)$. Since $u(x,t)=h_{D}(x,y',t)\in C^{\infty}(\Omega\times(0,+\infty))$ satisfies $(\partial_{t}-\triangle_{X})u(x,t)=0$, and $u(x,t)\geq 0$, the Bony's parabolic type strong maximum principle (see \cite{Bony1969} Theorem 3.2, also refer \cite{MB2014,Oleinik}) shows that $u(x,t)\equiv 0 $ for all $0<t\leq t'$ and all $x\in \Omega$. Now take a function $f\in C_{0}^{\infty}(\Omega)$ such that $f(y')\neq 0$, then we have $\lim_{t\to 0^+}\int_{\Omega}h_{D}(x,y',t)f(x)dx= f(y')$ and yet it is contradictory since $u(x,t)=h_{D}(x,y',t)\equiv 0$ for all $0<t\leq t'$ and all $x\in \Omega$. Hence we eventually obtain $h_{D}(x,y,t)>0$ for all $(x,y,t)\in \Omega\times\Omega\times(0,+\infty)$.\par
Let  $\Omega=\bigcup_{i=1}^{\infty}K_{i}$ with $K_{i}\subset K_{i+1}^{\circ}$. Here $K_{i}$ is compact set and $K_{i}^{\circ}$ the interior of $K_i$. Then we define a sequence of functions $f_i$ as
\[ f_{i}\in C_{0}^{\infty}(K_{i+1}^{\circ})\subset C_{0}^{\infty}(\Omega) ,~~ 0\leq f_{i}(x)\leq 1,~~\mbox{with } f_{i}(x)=\left\{
                                                   \begin{array}{ll}
                                                     1, & x\in\hbox{$K_{i}$}, \\[2mm]
                                                     0, & x\in\hbox{$(K_{i+1}^{\circ})^{c}$}.
                                                   \end{array}
                                                 \right.\]
It is easy to verify that $\lim_{i\to \infty}f_{i}(x)=\chi_{\Omega}(x)$, and $0\leq \chi_{K_{i}}(x)\leq f_{i}(x)\leq f_{i+1}(x)\leq \chi_{\Omega}(x)$. Using the weak maximum principle again, we obtain
\[  \int_{\Omega}f_{i}(y)h_{D}(x,y,t)dy\leq 1 \qquad\mbox{for all } i\in \mathbf{N}^{+}.\]
Then by Lebesgue's monotone convergence theorem, we have
\[\int_{\Omega}h_{D}(x,y,t)dy=\lim_{i\to\infty}\int_{\Omega}f_{i}(y)h_{D}(x,y,t)dy\leq 1. \]
Hence we complete the proof of \eqref{4-6}.\par Besides, if
$\overline{f}(x,t)$ is another solution of \eqref{4-8} with the same
initial condition $f_{0}$, then the weak maximum principle indicates
that the solution $f(x,t)-\overline{f}(x,t)$ of \eqref{4-8} must be
identically equal to $0$ since it vanishes on
$(\Omega\times\{0\})\cup (\partial\Omega\times [0,+\infty))$. This
leads to the uniqueness of $h_D(x,y,t)$.\par The arguments of all
above complete the proof of Proposition \ref{prop4-1}.
\end{proof}
\section{Diagonal asymptotic of sub-elliptic Dirichlet heat kernel}
 In this section, we  study the diagonal asymptotic behavior of sub-elliptic Dirichlet heat kernel of $\triangle_{X}$. First, by using the following proposition we can extend vector fields $X$ into whole space $\mathbb{R}^n$.
    \begin{proposition}
\label{prop5-1}
   Let $X=(X_{1},X_{2},\cdots,X_{m})$ be a system of $C^{\infty}$ vector fields defined in a bounded connected open set $W_{0}\subset \mathbb{R}^n$ and satisfying the condition (H) in $W_{0}$. Then, for any connected open sets $\Omega_{1}\subset\subset \Omega_{2}\subset \subset W_{0}$, there exists a new system of $C_{b}^{\infty}$ vector fields $X'=(Z_{1},Z_{2},\cdots,Z_{q})~(q=m+n)$, such that the vector fields $X'$ are defined in the whole space $\mathbb{R}^n$ and satisfy the  H\"{o}rmander's condition (H) in $\mathbb{R}^n$ (actually the vector fields $X'$ satisfy the uniform version of H\"{o}rmander's condition in $\mathbb{R}^n$, a detail proof will be given in Section 9,  Proposition \ref{prop9-1} below). Moreover
   \[ X'=\left\{
          \begin{array}{ll}
          (X_{1},X_{2},\cdots,X_{m},0,0,\cdots,0), & \hbox{in $\Omega_{1}$;} \\[2mm]
           (0,0,\cdots,0,\partial_{x_{1}},\partial_{x_{2}},\cdots,\partial_{x_{n}}) , & \hbox{in $\mathbb{R}^{n}\setminus \Omega_{2}$.}
          \end{array}
        \right.\]
Furthermore, denoting by $d_{X'},d$, respectively, the subunit metric induced by $X'$ in $\mathbb{R}^n$ and $X$ in $W_{0}$, then for any connected open set $\Omega\subset\subset\Omega_{1}$, $d_{X'}$ is equivalent to $d$ in $\Omega$, and $d_{X'}$ is equivalent to the Euclidean distance in $\mathbb{R}^{n}\setminus W_{0}$.
\end{proposition}
\begin{proof}
See Theorem 2.9 in \cite{Brandolini}.
\end{proof}

    Since $\overline{\Omega}$ is a compact subset of $W$, we can always find a bounded connected open set $W_{0}$ which has compact closure $\overline{W_{0}}$ such that $\overline{\Omega}\subset W_{0}\subset \overline{W_{0}}\subset W$. Also, there exists two connected open sets $\Omega_{1},\Omega_{2}$ such that
$\Omega\subset\subset\Omega_{1}\subset\subset\Omega_{2}\subset\subset W_{0}$.
Therefore, from Proposition \ref{prop5-1}, we get a system of $C_{b}^{\infty}$ vector fields $X'=(Z_{1},Z_{2},\cdots,Z_{q})$ which is an
extension of vector fields $X$ in $\mathbb{R}^n$ and satisfy the uniform H\"{o}rmander's condition. Let $\triangle_{X'}=-\sum_{j=1}^{q}Z_{j}^{*}Z_{j}$ be the sub-elliptic operator given by the vector fields $X'$, then $\triangle_{X'}=\triangle_{X}$ on $\Omega_{1}$ which is a neighborhood of $\overline{\Omega}$. For the sub-elliptic operator $\triangle_{X'}$ in $\mathbb{R}^{n}$, by the results in \cite{Brandolini,Kusuoka1988},
we know it has a global heat kernel $h(x,y,t)$ defined on $\mathbb{R}^n\times\mathbb{R}^n\times(0,+\infty)$ such that
\begin{equation}
\label{5-1}
  \left\{
    \begin{array}{ll}
   (\partial_{t}-\triangle_{X'})h(x,y,t)=0   , & \hbox{$\forall~(x,y,t)\in \mathbb{R}^n\times\mathbb{R}^n\times(0,+\infty)$;} \\[2mm]
     \lim\limits_{t\to 0^{+}}h(x,y,t)=\delta_{x}(y), &
    \end{array}
  \right.
\end{equation}
and also satisfies the following properties
\begin{equation}
\label{5-2}
\begin{aligned}
h(x,y,t)\geq 0~~ \mbox{for all}~~ (x,y,t)\in \mathbb{R}^{n}\times\mathbb{R}^{n}\times(0,+\infty),\\
\int_{\mathbb{R}^{n}}h(x,y,t)dy\leq 1~~ \mbox{for all}~~ (x,t)\in \mathbb{R}^{n}\times(0,+\infty).
\end{aligned}
\end{equation}
\begin{equation}
\label{5-3}
  h(x,y,t+s)=\int_{\mathbb{R}^{n}}h(x,z,t)h(z,y,s)dz,
\end{equation}
\begin{equation}
\label{5-4}
  h(x,y,t)=h(y,x,t).
\end{equation}
Meanwhile, the hypoellipticity of $\partial_{t}-\triangle_{X'}$ implies that $h(x,y,t)\in C^{\infty}(\mathbb{R}^n\times\mathbb{R}^n\times(0,+\infty))$.\par
 For the global heat kernel $h(x,y,t)$, we recall the asymptotic result constructed by Takanobu
 \cite{Takanobu1988}. Other similar results were also obtained by Ben Arous and L\'{e}andre \cite{Ben1989,Ben1990,Ben1991}.
\begin{proposition}
\label{prop5-2}
For the global heat kernel $h(x,y,t)$ of the sub-elliptic operator $\triangle_{X'}=-\sum_{j=1}^{q}Z_{j}^{*}Z_{j}$, there exists a sequence of real measurable functions $\{c_{i}(x)\}_{i=1}^{\infty}$ defined in $\mathbb{R}^n$, such that $h(x,x,t)$ has following asymptotic formula for $t\to 0^{+}$
\begin{equation}
\label{5-5}
  h(x,x,t)\sim t^{-\frac{\nu(x)}{2}}\left(\sum_{j=0}^{\infty}c_{j}(x)t^{j}\right),
\end{equation}
where $c_{0}(x)>0$ for all $x\in \mathbb{R}^n$, $\nu(x)$ is the pointwise homogeneous dimension at $x$.
\end{proposition}
\begin{proof}
See Theorem 6.8 in \cite{Takanobu1988}.
\end{proof}

   We also need the following Gaussian bounds of global heat kernel
    $h(x,y,t)$ which was proved by Kusuoka and Stroock in \cite{Kusuoka1987,Kusuoka1988} and
     was also generalized by Brandolini, Bramanti and Lanconelli et al \cite{Brandolini} to more general sub-elliptic operators. The similar results over compact manifolds was constructed by
     Jerison and S\'{a}nchez-Calle in \cite{Jerison1986}.
\begin{proposition}
\label{prop5-3}
For the global heat kernel $h(x,y,t)$ of the sub-elliptic operator $\triangle_{X'}=-\sum_{j=1}^{q}Z_{j}^{*}Z_{j}$,
there exist positive constants $A_{1},A_{2},B_{1},B_{2}$ such that for all $(x,y,t)\in \mathbb{R}^{n}\times\mathbb{R}^{n}\times (0,1]$, we have
\begin{equation}
\label{5-6}
  \frac{A_{1}}{|B_{d_{X'}}(x,\sqrt{t})|}\exp\left(\frac{-B_{1}d_{X'}^2(x,y)}{t} \right)\leq h(x,y,t)\leq   \frac{A_{2}}{|B_{d_{X'}}(x,\sqrt{t})|}\exp\left(\frac{-B_{2}d_{X'}^2(x,y)}{t} \right),
\end{equation}
 where $B_{d_{X'}}(x,r)=\{y\in \mathbb{R}^{n}|d_{X'}(x,y)<r\}$.
\end{proposition}
\begin{proof}
See \cite{Kusuoka1987}, \cite{Kusuoka1988} and \cite{Brandolini}.
\end{proof}

Then, we have the following diagonal asymptotic result of sub-elliptic Dirichlet heat kernel.
\begin{proposition}
\label{prop5-4} Let $h_{D}(x,y,t)$ be the sub-elliptic Dirichlet
heat kernel of $\triangle_{X}$ on $\Omega$. Then there exists a
non-negative measurable function $\gamma_{0}$ on $\overline{\Omega}$
which satisfies $\gamma_{0}(x)>0$ for all $x\in \Omega$, such that
\begin{equation}
\label{5-7}
  \lim_{t\to 0^{+}}t^{\frac{\nu(x)}{2}}h_{D}(x,x,t)=\gamma_{0}(x)~~\mbox{for all}~~ x\in\overline{\Omega}.
\end{equation}
\end{proposition}
\begin{proof}
Let $X'$ be a global extension of $X$ in $\mathbb{R}^n$ and $h(x,y,t)$ be the corresponding global heat kernel in $\mathbb{R}^n$. Given
 \[ E(x,y,t):=\left\{
              \begin{array}{ll}
             h(x,y,t)-h_{D}(x,y,t), & \hbox{$t>0$;} \\[2mm]
                0, & \hbox{$t\leq 0$,}
              \end{array}
            \right.
\]
it follows from \eqref{4-3} and \eqref{5-1} that for any $\varphi\in
C_{0}^{\infty}(\Omega)$, we have
\[ \lim_{t\to 0^{+}}\int_{\Omega}E(x,y,t)\varphi(x)dx=0. \]
Similar to the arguments in the proof of \eqref{4-6}, we have that
\begin{equation}\label{5-8}
E(x,y,t)\geq 0~~\mbox{for all}~(x,y,t)\in \Omega\times \Omega\times(0,+\infty).
\end{equation}
Also, it is easy to show that for any fixed $y\in \Omega$, $u_{y}(x,t)=E(x,y,t)$ is locally integrable on $\Omega\times \mathbb{R}$ and $u_{y}(x,t)$ satisfies
\[ (\partial_{t}-\triangle_{X})u_{y}(x,t)=0~~ \mbox{for all}~~ (x,t)\in \Omega\times \mathbb{R} \]
in the sense of distribution. Then the hypoellipticity of
$\partial_{t}-\triangle_{X}$ implies that for any fixed $y\in
\Omega$, $u_{y}(x,t)=E(x,y,t)\in
C^{\infty}(\Omega\times\mathbb{R})$. Moreover, for any $x\in
\overline{\Omega}$ we have
\[ \lim_{t\to 0^{+}}E(x,y,t)=E(x,y,0)=0. \]
Now, for any fixed $y\in \Omega$, we know that $E(x,y,t)\in C^{\infty}(\Omega\times (0,+\infty))\cap C(\overline{\Omega
}\times [0,+\infty))$. Then, by using the weak maximum principle, Proposition \ref{prop5-3} and \eqref{2-9}, we have for sufficient small $t\ll 1$ and all $x\in \Omega$
\begin{align*}
 E(x,y,t)&\leq \max_{z\in \partial\Omega,0\leq s\leq t}E(z,y,s)\\
&\leq \max_{z\in \partial\Omega,0\leq s\leq t}\left[\frac{A_{2}}{|B_{d_{X'}}(y,\sqrt{s})|}\exp\left(-\frac{B_{2}d_{X'}^2(z,y)}{s} \right)\right]\\
&\leq \max_{z\in \partial\Omega,0\leq s\leq t} Cs^{-\frac{n}{2\epsilon_{0}}}\exp\left(-\frac{B_{2}d_{X'}^2(z,y)}{s} \right)\\
&\leq \max_{0\leq s\leq t}C\cdot s^{-\frac{n}{2\epsilon_{0}}}\exp\left(-\frac{B_{2}\text{dist}_{d_{X'}}^{2}(y,\partial\Omega)}{s} \right),
\end{align*}
where $C$ is a positive constant depends on $A_{2}$ and $\overline{\Omega}$, $\epsilon_{0}>0$ is the constant in Fefferman and Phong's estimate \eqref{2-9}, $\text{dist}_{d_{X'}}(y,\partial\Omega):=\inf\{d_{X'}(x,y)|x\in \partial\Omega\}$. Now, we define $C_{2}:=C_{2}(y)=B_{2}\text{dist}_{d_{X'}}^2(y,\partial\Omega)>0$. Observe the function $g(s)=s^{-\frac{n}{2\epsilon_{0}}}e^{-\frac{C_{2}}{s}}$ satisfying $\lim_{s\to 0^{+}}g(s)=\lim_{s\to +\infty}g(s)=0$, and $g'(s)=e^{-\frac{C_{2}}{s}}s^{-\frac{n}{2\epsilon_{0}}-2}\left(C_{2}-\frac{n}{2\epsilon_{0}}s \right)$. Therefore, $g(s)$ can only attain its maximum value at $s=\frac{2\epsilon_{0}C_{2}}{n}>0$. Then, for any fixed $y\in \Omega$ and sufficient small $0<t\ll 1$, we have
\[ 0\leq E(x,y,t)\leq C\cdot t^{-\frac{n}{2\epsilon_{0}}}\cdot \exp\left(-\frac{C_{2}(y)}{t}\right)~~\mbox{for all}~~x\in \Omega.  \]
In particular, taking $x=y\in \Omega$, we have
\begin{equation}\label{5-9}
  0\leq E(x,x,t)\leq C\cdot t^{-\frac{n}{2\epsilon_{0}}}\cdot \exp\left(-\frac{C_{2}(x)}{t}\right)~~\mbox{for sufficient small}~t>0.
\end{equation}
Consequently
\[ 0\leq \lim_{t\to 0^{+}}t^{\frac{\nu(x)}{2}}E(x,x,t)\leq \lim_{t\to 0^{+}}Ct^{\frac{\nu(x)}{2}-\frac{n}{2\epsilon_{0}}}\cdot \exp\left(-\frac{C_{2}(x)}{t}\right)=0. \]
Thus, by Proposition \ref{prop5-2}, there exists a measurable function $c_{0}(x)$ in $\mathbb{R}^n$ such that
\[ \lim_{t\to 0^{+}}t^{\frac{\nu(x)}{2}}h_{D}(x,x,t)=\lim_{t\to 0^{+}}t^{\frac{\nu(x)}{2}}h(x,x,t)=c_{0}(x)>0~~ \mbox{for all}~~ x\in\Omega.\]
\par
We then show that the value of function $c_{0}(x)$ at each point $x\in \Omega$ is independent of the extension of vector fields $X'$. If $\widetilde{X}$ is another global extension of $X$ in $\mathbb{R}^{n}$, by the same approach, we also have
\begin{equation}
\label{5-10}
  0\leq t^{\frac{\nu(x)}{2}}(\widetilde{h}(x,x,t)-h_{D}(x,x,t))\leq \widetilde{C}t^{\frac{\nu(x)}{2}-\frac{n}{2\widetilde{\epsilon_{0}}}}\cdot\exp\left(-\frac{\widetilde{C}_{2}(x)}{t} \right)~~\mbox{for sufficient small}~t>0,~x\in \Omega.
\end{equation}
Here $\widetilde{h}(x,y,t)$ is the global heat kernel corresponding
with vector fields $\widetilde{X}$, $\widetilde{\epsilon_{0}}$ is a
positive constant depends on $\widetilde{X}$.
$\widetilde{C}_{2}(x)=\widetilde{B}_{2}\text{dist}_{d_{\widetilde{X}}}^2(x,\partial\Omega)$
is a positive constant depends on $x$ and the subunit metric induced
by $\widetilde{X}$. It follows from \eqref{5-9} and \eqref{5-10}
 that for sufficient small $t$ and all $x\in \Omega$, we have
\[ t^{\frac{\nu(x)}{2}}|h(x,x,t)-\widetilde{h}(x,x,t)|\leq \left[Ct^{\frac{\nu(x)}{2}-\frac{n}{2\epsilon_{0}}}\cdot \exp\left(-\frac{C_{2}(x)}{t}\right)+ \widetilde{C}t^{\frac{\nu(x)}{2}-\frac{n}{2\widetilde{\epsilon_{0}}}}\cdot\exp\left(-\frac{\widetilde{C}_{2}(x)}{t} \right)  \right].\]
Thus
\[ \lim_{t\to 0^{+}}t^{\frac{\nu(x)}{2}}h(x,x,t)=\lim_{t\to 0^{+}}t^{\frac{\nu(x)}{2}}\widetilde{h}(x,x,t)\qquad \mbox{for all}~~ x\in \Omega.\]
That implies the value of function $c_{0}(x)$ at each point $x\in \Omega$ is independent of the way of extension.\par
 Finally, we take
\[ \gamma_{0}(x)=\left\{
                   \begin{array}{ll}
                    c_{0}(x) , & \hbox{$x\in \Omega$;} \\[2mm]
                     0, & \hbox{$x\in \partial\Omega$.}
                   \end{array}
                 \right.\]
Then we obtain
\[ \lim_{t\to 0^{+}}t^{\frac{\nu(x)}{2}}h_{D}(x,x,t)=\gamma_{0}(x) ~~\mbox{for all}~~ x\in \overline{\Omega}.\]

\end{proof}

\section{Proofs of Theorem \ref{thm1}, Theorem \ref{thm2} and Theorem \ref{thm3}}

\subsection{Proof of Theorem \ref{thm1}.}
\begin{proof}
By the semi-group property of $h_{D}(x,y,t)$ in \eqref{4-5}, we have
\[ h_D(x,y,t)=\int_{\Omega}h_D(x,z,s)h_D(z,y,t-s)dz,\qquad \mbox{for } 0<s< t. \]
Since $h_D(x,y,t)=h_D(y,x,t)$, then we obtain
\begin{equation}\label{6-1}
h_D(x,x,2t)=\int_{\Omega}(h_D(x,z,t))^2dz, \qquad  \mbox{for all } t>0.
\end{equation}
Moreover
\begin{equation}\label{6-2}
\begin{split}
\frac{\partial}{\partial t}\int_{\Omega}(h_D(x,z,t))^2dz&=2\int_{\Omega}h_D(x,z,t)\triangle_{X}^{z}h_D(x,z,t)dz\\
&=-2\int_{\Omega}|X_{z}h_D(x,z,t)|^2dz\\
&\leq-2C\left(\int_{\Omega}|h_{D}(x,z,t)|^{\frac{2\tilde{\nu}}{\tilde{\nu}-2}}dz\right)^
{\frac{\tilde{\nu}-2}{\tilde{\nu}}}.
\end{split}
\end{equation}
The last inequality applies the weighted Sobolev inequality (Proposition \ref{pro2-3})
, which is valid since for any fixed $(x,t)\in \Omega\times(0,+\infty)$, $h_D(x,y,t)\in H_{X,0}^{1}(\Omega)$ with respect to  $y$.
\par

Now, it follows from \eqref{4-6} that
\[ \int_{\Omega}h_D(x,z,t)dz=\int_{\Omega}|h_D(x,z,t)|dz\leq 1. \]
Then the H\"{o}lder's inequality yields
\begin{equation}\label{6-3}
\left(\int_{\Omega}|h_D(x,z,t)|^{\frac{2\tilde{\nu}}{\tilde{\nu}-2}}dz\right)^
{\frac{\tilde{\nu}-2}{\tilde{\nu}}}\geq \left(\int_{\Omega}h_D(x,z,t)^2dz \right)^{\frac{2+\tilde{\nu}}{\tilde{\nu}}}.
\end{equation}
Hence \eqref{6-1},\eqref{6-2} and \eqref{6-3} give
\begin{equation}\label{6-4}
\frac{\partial}{\partial t}h_D(x,x,2t)+2C\cdot h_D(x,x,2t)^{\frac{2+\tilde{\nu}}{\tilde{\nu}}}\leq 0.
\end{equation}
For any fixed $x\in \Omega$, take $f(t):=h_D(x,x,2t)$ with $t>0$. The positivity of $h_{D}(x,y,t)$ implies that $f(t)>0$. Then it follows from \eqref{5-7} and \eqref{6-4} that
\[ \lim_{t\to 0^{+}}f(t)=+\infty,~~\mbox{and}~~\frac{f'(t)}{f(t)^{1+\frac{2}{\tilde{\nu}}}}\leq -2C,~~\mbox{for all } t>0. \]
Let
\[ g(t):=-\frac{\tilde{\nu}}{2}(f(t))^{-\frac{2}{\tilde{\nu}}}.\]
Then
\begin{equation}\label{6-5}
g'(s)=\frac{f'(s)}{f(s)^{1+\frac{2}{\tilde{\nu}}}}\leq -2C,~~\mbox{for }  s>0.
\end{equation}
Now integrating $g'(s)$ on $(\varepsilon,t)$ for any $t>0$ and $0<\varepsilon<t$, we obtain from \eqref{6-5} that
\begin{equation}\label{6-6}
  g(t)-g(\varepsilon)=\int_{\varepsilon}^{t}g'(s)ds\leq -2C(t-\varepsilon).
\end{equation}
Since $\lim_{t\to 0^{+}}f(t)=+\infty$, we know that $\lim_{t\to 0^{+}}g(t)=0$.
Letting $\varepsilon\to 0^{+}$ in \eqref{6-6}, we get
\[ g(t)\leq -2Ct. \]
Consequently
\[  h_{D}(x,x,2t)=f(t)\leq \left(\frac{4C}{\tilde{\nu}}t \right)^{-\frac{\tilde{\nu}}{2}}, ~~\mbox{for all } t>0.\]
Hence, we conclude that
\begin{equation}\label{6-7}
\quad h_D(x,x,t)\leq \left(\frac{2C}{\tilde{\nu}}t \right)^{-\frac{\tilde{\nu}}{2}},~~\mbox{for all } t>0.
\end{equation}
The upper bound estimate \eqref{1-7} of sub-elliptic Dirichlet heat kernel is proved, where $C$ is the Sobolev constant in \eqref{2-4}. This completes the proof of Theorem \ref{thm1}.
\end{proof}
\subsection{Proof of Theorem \ref{thm2}.}
\begin{proof}
Proposition \ref{prop4-1} gives us the following:
\begin{equation}\label{6-8}
h_{D}(x,y,t)=\sum_{i=1}^{\infty}e^{-\lambda_{i}t}\phi_{i}(x)\phi_{i}(y), \mbox{ for all }t>0.
 \end{equation}
It follows from Theorem \ref{thm1} that
\begin{equation}\label{6-9}
 h_D(x,x,t)\leq \frac{C}{t^{\frac{\tilde{\nu}}{2}}},~~ \mbox{ for all } t>0.
\end{equation}
Then, combining \eqref{6-8} and \eqref{6-9}, we get
\begin{equation}
\label{6-10}
\sum_{i=1}^{k}e^{-\lambda_{i}t}\phi_{i}^2(x)\leq \frac{C}{t^{\frac{\tilde{\nu}}{2}}}~~ \mbox{ for any }k\geq 1 \mbox{ and } t>0.
\end{equation}
Integrating \eqref{6-10} with respect to $x$ on $\Omega$ and using the fact $\int_{\Omega}\phi_{i}^2(x)dx=1$, we obtain
\begin{equation}\label{6-11}
  \sum_{i=1}^{k}e^{-\lambda_{i}t}\leq \frac{C|\Omega|}{t^{\frac{\tilde{\nu}}{2}}}.
\end{equation}
Since $x\mapsto e^{-x}$ is a convex function, then  \eqref{6-11} implies that
\begin{equation}\label{6-12}
ke^{-\frac{t}{k}\sum_{i=1}^{k}\lambda_{i}}\leq \sum_{i=1}^{k}e^{-\lambda_{i}t}\leq \frac{C|\Omega|}{t^{\frac{\tilde{\nu}}{2}}}.
\end{equation}
Putting $t=\frac{k}{\sum_{i=1}^{k}\lambda_{i}}$ into \eqref{6-12}, then
\begin{equation}\label{6-13}
\sum_{i=1}^{k}\lambda_{k}\geq C_{1}\cdot k^{1+\frac{2}{\tilde{\nu}}} \quad \mbox{for any } k\geq 1.
\end{equation}
Here $C_{1}=(Ce|\Omega|)^{-\frac{2}{\tilde{\nu}}}$ is a positive constant depending on  $\Omega$ and $\tilde\nu$.\par

The proof of the Theorem \ref{thm2} is now complete.
\end{proof}
\subsection{Proof of Theorem \ref{thm3}.}
We use the following Tauberian theorem to prove Theorem \ref{thm3}.
\begin{proposition}[Tauberian theorem]
\label{prop6-1}
Suppose that $\{\lambda_{n}\}_{n\in\mathbb{N}}$ is a sequence of positive real numbers, and for every $t>0$ the
series
\begin{equation}\label{6-14}
\sum_{n=1}^{+\infty}e^{-\lambda_{n}t}<+\infty.
\end{equation}
Then for $r>0$ and $a\in \mathbb{R}$, the following two arguments are equivalent:
\begin{itemize}
  \item
   \begin{equation}\label{6-15}
  \lim_{t\to 0^+}t^{r}\sum_{n=1}^{+\infty}e^{-\lambda_{n}t}=a,
  \end{equation}
  \item
  \begin{equation}\label{6-16}
   \lim_{\lambda\to+\infty}\lambda^{-r}N(\lambda)=\frac{a}{\Gamma(r+1)},
   \end{equation}
\end{itemize}
where $N(\lambda)=\#\{n|~0<\lambda_{n}\leq\lambda\}$ for $\lambda>0$.
\end{proposition}
\begin{proof}
See Theorem 1.1 in \cite{Arendt2009}.
\end{proof}

\begin{proof}[Proof of Theorem \ref{thm3}]
From Proposition \ref{5-4}, we know that for the sub-elliptic Dirichlet heat kernel $h_{D}(x,y,t)$, there exists a non-negative function $\gamma_{0}$ on $\overline{\Omega}$ such that
\begin{equation}\label{6-17}
  \lim_{t\to 0^{+}}t^{\frac{\nu(x)}{2}}h_{D}(x,x,t)=\gamma_{0}(x)>0~~\mbox{for all}~~ x\in \Omega.
\end{equation}
Hence, \eqref{6-17} implies
\begin{equation}\label{6-18}
 \lim_{t\to 0^{+}}t^{\frac{\tilde{\nu}}{2}}h_{D}(x,x,t)=\lim_{t\to 0^{+}}t^{\frac{\tilde{\nu}-\nu(x)}{2}}\cdot \lim_{t\to 0^{+}}t^{\frac{\nu(x)}{2}}h_{D}(x,x,t)=\left\{
                                         \begin{array}{ll}
                                          \gamma_{0}(x) , & \hbox{$\nu(x)=\tilde{\nu}$;} \\[2mm]
                                           0, & \hbox{$\nu(x)<\tilde{\nu}$.}
                                         \end{array}
                                       \right.
\end{equation}
Let $H=\{x\in \Omega|\nu(x)=\tilde{\nu}\}$ and $\chi_{H}$ be the
characteristic function of $H$. We can derive from \eqref{6-18} that
\begin{equation}\label{6-19}
\lim_{t\to 0^{+}}t^{\frac{\tilde{\nu}}{2}}h_{D}(x,x,t)=\gamma_{0}(x)\cdot\chi_{H}(x),~~\mbox{for all}~~ x\in \overline{\Omega}.
\end{equation}
According to Theorem \ref{thm1}, we have
\begin{equation}\label{6-20}
0\leq  t^{\frac{\tilde{\nu}}{2}}h_{D}(x,x,t)\leq C,~~\mbox{for all}~~ x\in \overline{\Omega},~ t>0.
\end{equation}
Combining \eqref{6-19} and \eqref{6-20}, it follows from the
Lebesgue's dominant  convergence theorem that
 \begin{equation}\label{6-21}
  \lim_{t\to 0^{+}}t^{\frac{\tilde{\nu}}{2}}\int_{\Omega}h_D(x,x,t)dx=\int_{\Omega}\lim_{t\to
  0^{+}}t^{\frac{\tilde{\nu}}{2}}h_D(x,x,t)dx=\int_{H}\gamma_{0}(x)dx<+\infty.
\end{equation}
Here $\gamma_{0}(x)>0$ for any $x\in H=\{x\in \Omega|~\nu(x)=\tilde{\nu}\}$.\par
On the other hand, from Proposition \ref{prop4-1} we get
\begin{equation}\label{6-22}
\sum_{k=1}^{\infty}e^{-\lambda_{k}t}=\int_{\Omega}h_D(x,x,t)dx<+\infty,~~\mbox{for all}~~ t>0.
\end{equation}
It follows from \eqref{6-21} and \eqref{6-22} that
\begin{equation}\label{6-23}
\lim_{t\to 0^{+}}t^{\frac{\tilde{\nu}}{2}}\cdot \sum_{k=1}^{\infty}e^{-\lambda_{k}t}=\int_{H}\gamma_{0}(x)dx.
\end{equation}
Then, by using the Proposition \ref{prop6-1}, we obtain
\begin{equation}\label{6-24}
\lim_{\lambda\to+\infty}N(\lambda)\cdot \lambda^{-\frac{\tilde{\nu}}{2}}=\frac{1}{\Gamma\left(\frac{\tilde{\nu}}{2}+1 \right)}\int_{H}\gamma_{0}(x)dx,
\end{equation}
where $N(\lambda)=\#\{k|~0<\lambda_{k}\leq \lambda\}$.

Taking $\lambda=\lambda_{k}$, since $\lambda_{k}\to +\infty$ as $k\to +\infty$, then \eqref{6-24} implies $N(\lambda_{k})=k+o(\lambda_{k}^{\frac{\tilde{\nu}}{2}})$ as $k\to +\infty$. Hence, we can also deduce from \eqref{6-24} that
\begin{equation}\label{6-25}
  \lim_{k\to +\infty}k\cdot \lambda_{k}^{-\frac{\tilde{\nu}}{2}}=\frac{1}{\Gamma\left(\frac{\tilde{\nu}}{2}+1 \right)}\int_{H}\gamma_{0}(x)dx.
\end{equation}
This straightforward implies that
\begin{itemize}
  \item If $|H|>0$,
\begin{equation}
\label{6-26}
 \lambda_{k}=\left(\frac{\Gamma\left(\frac{\tilde{\nu}}{2}+1\right)}
 {\int_{H}\gamma_{0}(x)dx}\right)^{\frac{2}{\tilde{\nu}}}\cdot
 k^{\frac{2}{\tilde{\nu}}}+o(k^{\frac{2}{\tilde{\nu}}}),~~\mbox{ as } k\to +\infty.
\end{equation}
\item If $|H|=0$,
  \begin{equation}\label{6-27}
    \lim_{k\to+\infty}\frac{k^{\frac{2}{\tilde{\nu}}}}{\lambda_{k}}=0.
  \end{equation}
\end{itemize}
Theorem \ref{thm3} is proved.
\end{proof}

\section{Proofs of Theorem \ref{thm4} and Theorem \ref{thm5}}
\subsection{Proof of Theorem \ref{thm4}.}
We shall use the generalization of an  approach in \cite{Laptev1997}
to give the proof of Theorem \ref{thm4}. First, we prove the
following proposition.
\begin{proposition}\label{pro7-1}
If $X=(X_{1},X_{2},\ldots,X_{m})$  satisfy the assumptions in Theorem \ref{thm4}, then for any $\lambda>0$, we have
\begin{equation}\label{7-1}
 \sum_{k=1}^{\infty}(\lambda-\lambda_{k})_{+}\geq C(\lambda-\lambda_{1})_{+}^{1+\frac{n}{2}},
\end{equation}
where the constant $C>0$ is dependent on $X$ and $\Omega$, $\lambda_{k}$ is the $k^{th}$ Dirichlet eigenvalue of $-\triangle_{X}$ on $\Omega$, $(\lambda-\lambda_{k})_{+}=\lambda-\lambda_k$ if $\lambda>\lambda_k$ and $(\lambda-\lambda_{k})_{+}=0$ if $\lambda\leq\lambda_k$.
\end{proposition}
\begin{proof}[Proof of Proposition \ref{pro7-1}]
Let $\phi_{1},\phi_{2},\cdots$ be the orthonormal eigenfunctions of $-\triangle_{X}$ on $\Omega$ which corresponding to the Dirichlet eigenvalues $0<\lambda_{1}<\lambda_{2}\leq \cdots$. It is easy to verify that the functions
\[ \theta_{\xi}(x):=\phi_{1}(x)e^{-ix\cdot \xi},\quad \xi\in \mathbb{R}^n, \]
belong to the domain $D(\triangle_{X})$ of operator $\triangle_{X}$.
  Denote
  \begin{equation}\label{7-2}
    \beta:=\sup_{x\in\Omega}|\phi_{1}(x)|>0.
  \end{equation}
Then, if we let $\varphi_{\lambda}(t):=(\lambda-t)_{+}$, we have
\begin{align*}
\sum_{k=1}^{\infty}(\lambda-\lambda_{k})_{+}&=\sum_{k=1}^{\infty}\varphi_{\lambda}(\lambda_{k})\int_{\Omega}|\phi_{k}(x)|^2dx\\
&\geq \beta^{-2}\sum_{k=1}^{\infty}\varphi_{\lambda}(\lambda_{k})\int_{\Omega}|\phi_{1}(x)\phi_{k}(x)|^2dx\\
&=\beta^{-2}\sum_{k=1}^{\infty}\varphi_{\lambda}(\lambda_{k})\int_{\mathbb{R}^n}|\phi_{1}(x)\phi_{k}(x)|^2dx\\
&=\beta^{-2}(2\pi)^{-n}\sum_{k=1}^{\infty}\varphi_{\lambda}(\lambda_{k})\int_{\mathbb{R}^n}\left|\int_{\mathbb{R}^n}\phi_{k}(x)\theta_{\xi}(x)dx\right|^{2}d\xi.
\end{align*}
Let $E_{s}$ be the spectral projection of the self-adjoint operator
$-\triangle_{X}$. Then we obtain
\begin{align*}
\sum_{k=1}^{\infty}(\lambda-\lambda_{k})_{+}&\geq\beta^{-2}(2\pi)^{-n}\sum_{k=1}^{\infty}\varphi_{\lambda}(\lambda_{k})\int_{\mathbb{R}^n}\left|\int_{\mathbb{R}^n}\phi_{k}(x)\theta_{\xi}(x)dx\right|^{2}d\xi\\
&=\beta^{-2}(2\pi)^{-n}\int_{\mathbb{R}^{n}}\int_{0}^{+\infty}\varphi_{\lambda}(s)d(E_{s}\theta_{\xi},\theta_{\xi})d\xi.
\end{align*}
Clearly here we have
\[ \int_{0}^{+\infty}d(E_{s}\theta_{\xi},\theta_{\xi})=\|\theta_{\xi}\|_{L^2(\Omega)}^2=\|\phi_{1}\|_{L^2(\Omega)}^2=1.\]
Since $\varphi_{\lambda}(t)$ is a convex function, then we use the Jensen inequality to deduce
\begin{align*}
\sum_{k=1}^{\infty}(\lambda-\lambda_{k})_{+}&\geq   \beta^{-2}(2\pi)^{-n}\int_{\mathbb{R}^{n}}\int_{0}^{+\infty}\varphi_{\lambda}(s)d(E_{s}\theta_{\xi},\theta_{\xi})d\xi\\
&\geq \beta^{-2}(2\pi)^{-n}\int_{\mathbb{R}^{n}}\varphi_{\lambda}\left(\int_{0}^{+\infty}sd(E_{s}\theta_{\xi},\theta_{\xi})\right)d\xi.
\end{align*}
A simple calculation gives
\[ \int_{0}^{+\infty}sd(E_{s}\theta_{\xi},\theta_{\xi})=(-\triangle_{X}\theta_{\xi},\theta_{\xi})_{L^2(\Omega)}=\int_{\Omega}|X\theta_{\xi}(x)|^2dx.\]
On the other hand, for each $X_{j}=\sum_{k=1}^{n}a_{jk}(x)\partial_{x_{k}}$, we introduce a vector which corresponding to the differential operator $X_{j}$ by
\[ X_{j}I(x):=(a_{j1}(x),a_{j2}(x),\cdots,a_{jn}(x)).\]
 Then we can deduce that
\[X_{j}(\theta_{\xi}(x))=X_{j}(\phi_{1}(x)e^{-ix\cdot\xi})=e^{-ix\cdot\xi}[(X_{j}\phi_{1})-i\phi_{1}(x)\left \langle  X_{j}I(x), \xi \right \rangle_{\mathbb{R}^{n}}], \]
where   $\left \langle  X_{j}I(x), \xi \right \rangle_{\mathbb{R}^{n}}=\sum_{k=1}^{n}a_{jk}(x)\xi_{k}$ is the inner product of vector $X_{j}I(x)$ and $\xi$ in $\mathbb{R}^n$. Thus,
\[ |X_{j}(\theta_{\xi}(x))|^2=|X_{j}\phi_{1}|^2+\phi_{1}^2(x)\left \langle  X_{j}I(x), \xi \right \rangle_{\mathbb{R}^{n}}^2.     \]
Then, we have
\begin{align*}
\int_{\Omega}|X\theta_{\xi}(x)|^2dx&=\sum_{j=1}^{m}\int_{\Omega}|X_{j}\theta_{\xi}(x)|^2dx\\
&=\int_{\Omega}|X\phi_{1}|^2dx+\int_{\Omega}\phi_{1}^2(x)\sum_{j=1}^{m}\left \langle  X_{j}I(x), \xi \right \rangle_{\mathbb{R}^{n}}^2dx\\
&=\lambda_{1}+\int_{\Omega}\phi_{1}^2(x)\sum_{j=1}^{m}\left \langle  X_{j}I(x), \xi \right \rangle_{\mathbb{R}^{n}}^2dx\\
&\leq \lambda_{1}+\int_{\Omega}\phi_{1}^2(x)\left(\sum_{j=1}^{m}|X_{j}I(x)|^2\right)\cdot|\xi|^2dx
\end{align*}
Recall that $X=(X_{1},X_{2},\cdots,X_{m})$ are  $C^{\infty}$ vector
fields defined on the compact domain  $\overline{\Omega}$, then we
have
\[ \int_{\Omega}|X\theta_{\xi}(x)|^2dx\leq \lambda_{1}+M\int_{\Omega}\phi_{1}^2(x)|\xi|^2dx=\lambda_{1}+M|\xi|^2,\]
where $M=\sup_{x\in\Omega}\left(\sum_{j=1}^{m}|X_{j}I(x)|^2\right)<+\infty$. Observe that $\varphi_{\lambda}(t)$ is decrease with respect to $t$, hence we obtain
\begin{align*}
\sum_{k=1}^{\infty}(\lambda-\lambda_{k})_{+}&\geq \beta^{-2}(2\pi)^{-n}\int_{\mathbb{R}^{n}}\varphi_{\lambda}\left(\int_{0}^{+\infty}sd(E_{s}\theta_{\xi},\theta_{\xi})\right)d\xi\\
&=\beta^{-2}(2\pi)^{-n}\int_{\mathbb{R}^{n}}\varphi_{\lambda}\left(\int_{\Omega}|X\theta_{\xi}(x)|^2dx\right)d\xi\\
&\geq \beta^{-2}(2\pi)^{-n}\int_{\mathbb{R}^{n}}\varphi_{\lambda}(\lambda_{1}+M|\xi|^2)d\xi\\
&=\beta^{-2}(2\pi)^{-n}\int_{\mathbb{R}^{n}}(\lambda-\lambda_{1}-M|\xi|^2)_{+}d\xi\geq C(\lambda-\lambda_{1})_{+}^{1+\frac{n}{2}},
\end{align*}
where the positive constant $C$ depends on $X$ and $\Omega$. The proof of Proposition \ref{pro7-1} is complete.

\end{proof}

\begin{proof}[Proof of Theorem \ref{thm4}]
Now, we take $\lambda=\lambda_{k}$ in Proposition \ref{pro7-1}. Then
we get
\begin{equation}\label{7-3}
  \sum_{j=1}^{k-1}(\lambda_{k}-\lambda_{j})\geq C(\lambda_{k}-\lambda_{1})_{+}^{1+\frac{n}{2}}.
\end{equation}
For $k\geq 2$, we have $\lambda_{k}>\lambda_{1}$, this implies $\frac{\lambda_{k}-\lambda_{j}}{\lambda_{k}-\lambda_{1}}\leq 1$ for $j=1,2,\ldots,k-1$. Hence, we have
\[ k-1\geq  \sum_{j=1}^{k-1}\frac{\lambda_{k}-\lambda_{j}}{\lambda_{k}-\lambda_{1}}\geq C(\lambda_{k}-\lambda_{1})^{\frac{n}{2}}.\]
Consequently
\[ \lambda_{k}\leq \tilde{C}\cdot (k-1)^{\frac{2}{n}}+\lambda_{1}~~\mbox{for all}~~ k\geq 1.\]
The proof of Theorem \ref{thm4} is complete.
\end{proof}
\subsection{Proof of Theorem \ref{thm5}.}

Combining Proposition \ref{prop5-3} with \eqref{5-8}, we obtain that
for Dirichlet heat kernel $h_{D}(x,y,t)$ of sub-elliptic operator
$\triangle_{X}=-\sum_{i=1}^{m}X_{i}^{*}X_{i}$, there exists $C_{2}>0$
such that
\begin{equation}\label{7-4}
  h_{D}(x,y,t)\leq \frac{C_{2}}{|B_{d_{X'}}(x,\sqrt{t})|}e^{-\frac{d_{X'}(x,y)^2}{C_{2}t}}
\end{equation}
holds for all $t\in (0,1), x,y \in \overline{\Omega}$. Here $B_{d_{X'}}(x,r)$ is the subunit ball induced by the Carnot-Carath\'{e}odory metric $d_{X'}(x,y)$ which depends on the extension $X'$. In particular, we have
\begin{equation}\label{7-5}
h_{D}(x,x,t)\leq \frac{C_{2}}{|B_{d_{X'}}(x,\sqrt{t})|}~~\mbox{for all}~~ x\in \overline{\Omega}, ~0<t<1.
\end{equation}
Integrating \eqref{7-5} with respect $x$ on $\Omega$, we obtain
\begin{equation}\label{7-6}
\int_{\Omega}h_{D}(x,x,t)dx\leq C_{2}\int_{\Omega}\frac{1}{|B_{d_{X'}}(x,\sqrt{t})|}dx~~\mbox{for all}~~  0<t<1.
\end{equation}
Now, by using Proposition \ref{pro2-6}, since $\overline{\Omega}$ is
 a compact subset of $\mathbb{R}^n$, there exists
$\delta_{0}=\delta_{0}(\overline{\Omega})>0$ and constants
$C_{3},C_{4}>0$ such that
\begin{equation} \label{7-7}
C_{3}\Lambda(x,r)\leq |B_{d_{X'}}(x,r)|\leq C_{4}\Lambda(x,r)~~\mbox{for all}~~ x\in \overline{\Omega},~0<r<\delta_{0}.
\end{equation}
Take $\delta_{1}=\min\{1,\delta_{0}^2\}$, by \eqref{7-6} and \eqref{7-7}  we have for a constant $C_5>0$
\begin{equation} \label{7-8}
\int_{\Omega}h_{D}(x,x,t)dx\leq C_{5}\int_{\Omega}\frac{1}{\Lambda(x,\sqrt{t})}dx~~\mbox{for all}~~ 0<t<\delta_{1}.
\end{equation}
On the other hand, the formula \eqref{2-11} gives
\begin{equation} \label{7-9}
\Lambda(x,\sqrt{t})=\sum_{I}|\lambda_{I}(x)|t^{\frac{d(I)}{2}}\geq \sum_{d(I)=n}|\lambda_{I}(x)|t^{\frac{d(I)}{2}}=t^{\frac{n}{2}}\sum_{d(I)=n}|\lambda_{I}(x)|.
\end{equation}
If the vector fields $X=(X_{1},X_{2},\cdots,X_{m})$ satisfy the condition (A) on $\Omega$, then from \eqref{1-13} we have that
\begin{equation} \label{7-10}
\int_{\Omega}\frac{dx}{\sum_{d(I)=n}|\lambda_{I}(x)|} \leq \int_{\Omega}\frac{dx}{\sum|\det(Y_{i_{1}},Y_{i_{2}},\cdots,Y_{i_{n}})(x)|}<+\infty,
\end{equation}
 where the second sum in \eqref{7-10} is over all $n$-combinations $ (Y_{i_{1}},Y_{i_{2}},\cdots,Y_{i_{n}})$ of set $\{X_{j}|1\leq j\leq m\}$.
Combining \eqref{7-8}, \eqref{7-9} and \eqref{7-10}, we get
\begin{equation} \label{7-11}
\int_{\Omega}h_{D}(x,x,t)dx\leq C_{6}t^{-\frac{n}{2}}~~\mbox{for all}~~0<t<\delta_{1},
\end{equation}
where $C_{6}=C_{5}\int_{\Omega}\frac{1}{\sum|\det(Y_{i_{1}},Y_{i_{2}},\cdots,Y_{i_{n}})(x)|}<+\infty$. Recall that the Dirichlet heat kernel $h_{D}(x,y,t)$ can be expanded by the following series which converges uniformly in $\overline{\Omega}\times\overline{\Omega}\times [a,+\infty)$ for any $a>0$,
\begin{equation}
\label{7-12}
 h_{D}(x,y,t)=\sum_{i=1}^{\infty}e^{-\lambda_{i}t}\phi_{i}(x)\phi_{i}(y) .
\end{equation}
From the fact $\int_{\Omega}\phi_{j}^2(x)dx=1$, we have for any $k\in \mathbb{N}^{+}$,
\begin{equation}\label{7-13}
\sum_{j=1}^{k}e^{-\lambda_{j}t}=\int_{\Omega}\sum_{j=1}^{k}e^{-\lambda_{j}t}\phi_{j}^2(x)dx\leq \int_{\Omega}h_{D}(x,x,t)dx~~\mbox{for all}~t>0.
\end{equation}
Hence
\begin{equation}\label{7-14}
\sum_{j=1}^{k}e^{-\lambda_{j}t}\leq C_{6}t^{-\frac{n}{2}}~~\mbox{for all}~~ k\in \mathbb{N}^{+}, 0<t<\delta_{1}.
\end{equation}
Since $x\mapsto e^{-x}$ is a convex function, then from \eqref{7-14} we have
\[ ke^{-\frac{t}{k}\sum_{i=1}^{k}\lambda_{i}}\leq \sum_{j=1}^{k}e^{-\lambda_{j}t}\leq C_{6}t^{-\frac{n}{2}}~~\mbox{for all}~~ k\in \mathbb{N}^{+}, 0<t<\delta_{1}.\]
Since $\lambda_{j}\leq \lambda_{j+1}$, then if we take $t=\frac{\delta_{1}}{2}\cdot \frac{k\lambda_{1}}{\sum_{i=1}^{k}\lambda_{i}}\in (0,\delta_{1})$, we can obtain
\[ ke^{-\frac{\delta_{1}\lambda_{1}}{2}}\leq C_{6}\left(\frac{\delta_{1}}{2}\cdot\frac{k\lambda_{1}}{\sum_{i=1}^{k}\lambda_{i}} \right)^{-\frac{n}{2}}.
\]
Therefore, we can conclude that
\[ \sum_{i=1}^{k}\lambda_{i}\geq \frac{\delta_{1}\lambda_{1}}{2}\cdot (C_{6}e^{\frac{\delta_{1}\lambda_{1}}{2}})^{-\frac{2}{n}}\cdot k^{1+\frac{2}{n}}~~\mbox{for all}~k\geq 1.\]
The proof of Theorem \ref{thm5} is complete.

\section{Some Examples}
In this section, as applications of Theorem \ref{thm2}--Theorem \ref{thm5}, we give some examples.
\begin{ex}[Kohn Laplacian $\triangle_{\mathbb{H}}$]
Let $(\mathbb{H}_{n},\circ)$ be the Heisenberg group  in $\mathbb{R}^{2n+1}$. Here $\circ$ is the group operation on the Heisenberg group $\mathbb{H}_{n}$ defined as follows:\par
 Given the two points
\[ \xi=(x_{1},x_{2},\cdots,x_{n},y_{1},y_{2},\cdots,y_{n},z)=(x,y,z)\in \mathbb{H}_{n},\qquad x,y\in \mathbb{R}^n,~~ z\in \mathbb{R}\]
and
\[ \xi'=(x_{1}',x_{2}',\cdots,x_{n}',y_{1}',y_{2}',\cdots,y_{n}',z')=(x,y,z)\in \mathbb{H}_{n},\qquad x',y'\in \mathbb{R}^n,~~ z'\in \mathbb{R}.\]
Then
\[ \xi'\circ \xi:=(x'+x,y'+y,z'+z-2(x'\cdot y-x\cdot y')), \]
where the point $\cdot$ stands for the inner product in $\mathbb{R}^n$.

Consider the Kohn Laplacian on Heisenberg group $\mathbb{H}_{n}\subset \mathbb{R}^{2n+1}$,
 \[\triangle_{\mathbb{H}}:=\sum_{j=1}^{n}(X_{j}^2+Y_{j}^2), \]
which is induced by the vector fields $X_{j}=\partial_{x_{j}}+2y_{j}\partial_{z},Y_{j}=\partial_{y_{j}}-2x_{j}\partial_{z}$ for $j=1,2,\cdots, n$. In this case, we know the condition (H) and (M) are permissible in $\mathbb{R}^{2n+1}$, with H\"ormander index $Q=2$ and M\'{e}tivier index $\nu=2n+2$.\par
 Let $\Omega\subset \mathbb{R}^{2n+1}$ be a bounded connected open set with non-characteristic smooth boundary for vector fields $X=(X_{1},\cdots,X_{n},Y_{1},\cdots,Y_{n})$.
For the Dirichlet eigenvalue problem \eqref{1-1} on $\Omega$,  Hansson and Laptev \cite{AM1994} have proved that
\begin{equation}\label{8-1}
  \lambda_{k}\geq \left(\frac{(2\pi)^{n+1}(n+1)^{n+2}}{2C_{n}(n+2)^{n+1}|\Omega|} \right)^{\frac{1}{n+1}}\cdot k^{\frac{1}{n+1}}~~\mbox{  for all }k\geq 1,
\end{equation}
where $C_{n}=\sum_{j_{1},\cdots,j_{n}\geq 0}(2(j_{1}+\cdots+j_{n})+n)^{-(n+1)}$.\par

Now by our estimation in Theorem \ref{thm2}, we get
\begin{equation}\label{8-2}
\sum_{j=1}^{k}\lambda_{j}\geq C \cdot k^{1+\frac{1}{n+1}}~\mbox{ for all }k\geq 1,
\end{equation}
where $C$ is a positive constant related to $X$ and $\Omega$. Furthermore, we can get an explicit constant $C$ via the comparison of Dirichlet heat kernel and global heat kernel. From the results in \cite{YV1998}, we know that $\triangle_{\mathbb{H}}$ has a non-negative global heat kernel $h(x,y,t)$ such that
\[ h(\xi,0,t)=\frac{1}{2(4\pi t)^{n+1}}\int_{-\infty}^{+\infty}\left(\frac{\theta}{\sinh \theta} \right)^{n}\exp\left(-\frac{iz\theta+r^2\theta\coth\theta}{4t}  \right)d\theta, \]
where $\xi=(x_{1},\cdots,x_{n},y_{1},\cdots,y_{n},z)\in \mathbb{H}_{n}, r^2=\sum_{i=1}^{n}(x_{i}^2+y_{i}^2)$. Since the invariance of the operator $\triangle_{\mathbb{H}}$ with respect to left translations, we have
\[ h(\xi,\xi',t)=h(-\xi'\circ \xi,0,t). \]
Moreover, we have that
\[ h(\xi,\xi,t)=h(0,0,t)=\frac{\alpha_{n}}{(4\pi t)^{n+1}}, \]
where $\alpha_{n}=\int_{0}^{+\infty}\left(\frac{\theta}{\sinh \theta}\right)^{n} d\theta$. Since $h_{D}(x,y,t)\leq h(x,y,t)$, we obtain
\begin{equation}\label{8-3}
  h_{D}(x,x,t)\leq \frac{\alpha_{n}}{(4\pi t)^{n+1}}~~\mbox{for all}~~ t>0, x\in \Omega.
\end{equation}
Therefore, for any $t>0$ we can deduce from \eqref{8-3} that
\begin{equation}\label{8-4}
  k\cdot \exp\left(-\frac{t}{k}\sum_{j=1}^{k}\lambda_{j} \right)\leq \sum_{j=1}^{k}e^{-\lambda_{j}t}\leq \frac{\alpha_{n}}{(4\pi t)^{n+1}}|\Omega|.
\end{equation}
In order to get a sharp constant $C$, we take $t=s\cdot \frac{k}{\sum_{j=1}^{k}\lambda_{j}}$ in \eqref{8-4}, where $s>0$
is a constant to be determined later. Then, we have
\begin{equation}\label{8-5}
\sum_{j=1}^{k}\lambda_{j}\geq \frac{4\pi}{\alpha_{n}^{\frac{1}{n+1}}|\Omega|^{\frac{1}{n+1}}}\cdot \frac{s}{e^{\frac{s}{n+1}}}\cdot k^{1+\frac{1}{n+1}}.
\end{equation}
Now, we let $g(s)=\frac{s}{e^{\frac{s}{n+1}}}$. It is easy to show that $g(n+1)=\max\limits_{s\in (0,+\infty)}g(s)=(n+1)e^{-1}$. Hence, if we put $s=n+1$ in \eqref{8-5}, we can get a lower bound with an explicit coefficient
\begin{equation}\label{8-6}
\sum_{j=1}^{k}\lambda_{j}\geq \frac{4\pi}{\alpha_{n}^{\frac{1}{n+1}}|\Omega|^{\frac{1}{n+1}}}\cdot (n+1)e^{-1}\cdot k^{1+\frac{1}{n+1}},
\end{equation}
where $\alpha_{n}=\int_{0}^{+\infty}\left(\frac{\theta}{\sinh\theta} \right)^{n} d\theta$.
\end{ex}
  As we can see that, for the H\"{o}rmander vector fields $X$ with $|H|>0$, Theorem \ref{thm2} and Theorem \ref{thm3} give the optimal estimates of Dirichlet eigenvalues. Here we shall give an example below in which the M\'{e}tivier's condition (M) will be not satisfied on $\Omega$, but the set  $H=\{x\in \Omega~|~\nu(x)=\tilde{\nu}\}$ has a strict positive measure.

\begin{ex}
Let $\Omega\subset \mathbb{R}^3$ be a bounded connected open set with smooth boundary $\partial\Omega$ such that $D(2):=\{(x_{1},x_{2},x_{3})\in \mathbb{R}^3||x_{i}|<2, i=1,2,3\}\subset\subset \Omega$. Given the vector fields $X_{1},X_{2},X_{3}$ defined in $\mathbb{R}^3$ such that
\[ X_{1}=\frac{\partial}{\partial x_{1}}-\frac{1}{2}x_{2}\frac{\partial}{\partial x_{3}},~~ X_{2}=\frac{\partial}{\partial x_{2}}+\frac{1}{2}x_{1}\frac{\partial}{\partial x_{3}},\]
\[ X_{3}=(\phi_{1}(x_{1},x_{2})+\phi_{2}(x_{3})+\phi_{3}(x_{3}))\frac{\partial}{\partial x_{3}}, \]
where

     \[ \phi_{1}(x_{1},x_{2})=\left\{
                            \begin{array}{ll}
                          e^{-(\log(\sqrt{x_{1}^2+x_{2}^2}-\frac{3}{2}))^2}    , & \hbox{$\sqrt{x_{1}^2+x_{2}^2}>\frac{3}{2}$;} \\[2mm]
                             0, & \hbox{ $\sqrt{x_{1}^2+x_{2}^2}\leq\frac{3}{2}$.}
                            \end{array}
                          \right.\]

      \[ \phi_{2}(x_{3})=\left\{
                          \begin{array}{ll}
                            e^{-(\log x_{3})^2}, & \hbox{$x_{3}\in (0,+\infty)$;} \\[2mm]
                            0, & \hbox{$x_{3}\in (-\infty,0]$.}
                          \end{array}
                        \right.\]
 and
       \[ \phi_{3}(x_{3})=\left\{
                               \begin{array}{ll}
                             e^{-(\log(-x_{3}-1))^{2}}    , & \hbox{ $x_{3}\in (-\infty,-1)$;} \\[2mm]
                                 0, & \hbox{ $x_{3}\in [-1,+\infty)$.}
                               \end{array}
                             \right.\]
From above assumptions, we can see that the vector fields $X=(X_{1},X_{2},X_{3})$ verify the H\"{o}rmander's condition on $\overline{\Omega}$ with H\"{o}mander index $Q=2$. If we denote $H$ by the set
\[ H:=\left\{(x_{1},x_{2},x_{3})\in \mathbb{R}^3\bigg|\sqrt{x_{1}^2+x_{2}^2}\leq \frac{3}{2}, -1\leq x_{3}\leq 0\right\}. \]
We know that $ H\subset\subset D(2)\subset\subset \Omega $. Then we have
\[ \nu_{1}(x)=\dim V_{1}(x)=\left\{
                              \begin{array}{ll}
                                2, & \hbox{$x\in H$;} \\[2mm]
                                3, & \hbox{$x\in \overline{\Omega}\setminus H$,}
                              \end{array}
                            \right.\]
and
\[ \nu_{2}(x)=\dim V_{2}(x)=\left\{
                              \begin{array}{ll}
                                3, & \hbox{$x\in H$;} \\[2mm]
                                3, & \hbox{$x\in \overline{\Omega}\setminus H$,}
                              \end{array}
                            \right.\]
Therefore
\[ \nu(x)=\sum_{j=1}^{2}j(\nu_{j}(x)-\nu_{j-1}(x))=\left\{
            \begin{array}{ll}
              4, & \hbox{if $x\in H$ ;} \\[2mm]
              3, & \hbox{if $x\in \overline{\Omega}\setminus H$ .}
            \end{array}
          \right.\]
The vector fields $X=(X_{1},X_{2},X_{3})$ do not satisfy the M\'{e}tivier's condition (M), but has generalized M\'{e}tivier index $\tilde{\nu}$ on $\Omega$, namely
\[  \tilde{\nu}=\max_{x\in \overline{\Omega}}\nu(x)=4.\]
 For the set  $H=\left\{(x_{1},x_{2},x_{3})\in \mathbb{R}^3\bigg|\sqrt{x_{1}^2+x_{2}^2}\leq \frac{3}{2}, -1\leq x_{3}\leq 0\right\}=\{x\in \Omega|\nu(x)=\tilde{\nu}\}$, we know that $|H|>0$.  If we consider the Dirichlet eigenvalue problem \eqref{1-1} for the sub-elliptic operator $\triangle_{X}=-\sum_{j=1}^{3}X_{j}^{*}X_{j}$ on $\Omega$, according to the Theorem \ref{thm2}, we get a lower bound for $\lambda_{k}$
\begin{equation}\label{8-7}
  \sum_{j=1}^{k}\lambda_{k}\geq C\cdot k^{1+\frac{2}{\tilde{\nu}}}=C\cdot k^{\frac{3}{2}}.
\end{equation}
Thus Theorem \ref{thm3} tells us this lower bound is optimal in sense of the order $k$ and there exists $C_1>0$ which is dependent on the vector fields $X=(X_1, X_2, X_3)$ and $\Omega$, such that $\lambda_{k}\sim  C_1 k^{\frac{1}{2}}$ as $k\to +\infty$.

\end{ex}

In the following example, we shall have $|H|=0$ and the condition (A) is satisfied.

\begin{ex}
For $n\geq 3$, let $G=(X_{1},X_{2},\cdots,X_{n-1},Y_{1},Y_{2},\cdots,Y_{n-1})$ be the vector fields defined on an open connected set $W\subset\mathbb{R}^{n}$, the Grushin operator induced by $G$ is defined as follows:
 \begin{equation}\label{8-8}
 \left\{
       \begin{array}{ll}
       \triangle_G:=\sum\limits_{j=1}^{n-1}(X_j^2+Y_j^2),\\[4mm]
       X_j=\frac{\partial}{\partial x_j},~~Y_j=x_j\frac{\partial}{\partial x_{n}},~~1\leq j\leq n-1,
        \end{array}
       \right.
 \end{equation}
Assume $\Omega\subset\subset W$ to be a  bounded connected open subset of $W$ with smooth
 boundary $\partial\Omega$ which is non-characteristic for $G$. Also $\Omega$ satisfies that
 $\Omega\cap\{(0,0,\cdots,0,x_{n})|x_{n}\in \mathbb{R}\}\neq \varnothing $, and the generalized M\'{e}tivier index $\tilde{\nu}=n+1$.
 If we let $x'=(x_{1},\cdots,x_{n-1})$ then  $dx=dx'dx_{n}$ and $\Omega\subset \Omega_{x'}\times\Omega_{x_{n}}$. Here $\Omega_{x'},\Omega_{x_{n}}$ are the projections of $\Omega$ in $\{(x_{1},\cdots,x_{n-1},0)\}$ and $\{(0,\cdots,0,x_{n})\}$. Recall $\Omega\cap\{(\mathbf{0},x_{n})\in\mathbb{R}^{n}\}\neq \varnothing$, that implies  $\mathbf{0}\in \Omega_{x'}$. Since $\Omega_{x'}$ is an open set, there exists $\delta>0$ such that $B(\mathbf{0},\delta)\subset \Omega_{x'}$. By a direct calculation, we know that
\[ \sum|\det(Y_{i_{1}},Y_{i_{2}},\cdots,Y_{i_{n}})|=(|x_{1}|+\cdots+|x_{n-1}|),\]
where the sum is over all $n$-combinations $(Y_{i_{1}},Y_{i_{2}},\cdots,Y_{i_{n}})$ of the set $\{X_{j},Y_{j}|1\leq j\leq n-1\}$.
Observe that
\[ |x'|=\sqrt{x_{1}^2+\cdots+x_{n-1}^2}\leq |x_{1}|+\cdots+|x_{n-1}|. \]
Thus, in order to verify the assumption (A), it suffices to prove the convergence of  integral
$\int_{\Omega_{x'}}\frac{dx'}{|x'|}$. Indeed,
we can obtain that
\begin{equation}\label{8-9}
\int_{\Omega_{x'}}\frac{dx'}{|x'|}=\int_{B(\mathbf{0},\delta)}\frac{dx'}{|x'|}+\int_{\Omega_{x'}\setminus B(\mathbf{0},\delta)}\frac{dx'}{|x'|}.
\end{equation}
We know that the second part in \eqref{8-9} is finite. Then, for the first part in \eqref{8-9}, we have
\[ \int_{B(\mathbf{0},\delta)}\frac{dx'}{|x'|}=\omega_{n-1}\int_{0}^{\delta}r^{n-3}dr=\omega_{n-1}\frac{\delta^{n-2}}{n-2}<+\infty. \]
Hence, by Theorem \ref{thm5}, we have the following estimate for $\lambda_{k}$
\begin{equation}\label{8-10}
\sum_{j=1}^{k}\lambda_{j}\geq C\cdot k^{1+\frac{2}{n}}~~\mbox{for all}~~ k\geq 1,
\end{equation}
where $C>0$ is some constant which depends on the vector fields $G$ and $\Omega$.

From the upper bound estimate of $\lambda_{k}$ in Theorem \ref{thm4} and the lower bound estimate \eqref{8-10}, we know that $\lambda_{k}\approx k^{\frac{2}{n}}$ as $k\to +\infty$ in this example, which indeed improves the results for this Grushin type sub-elliptic operator in \cite{CCD} and \cite{chenluo}.

\end{ex}

Finally, we give an example for Grushin type vector fields, in which $|H|=0$ but the condition (A) is not satisfied. In this case, we can see the increase order of $k$ for $\lambda_{k}$  may smaller than $k^{\frac{2}{n}}$.

\begin{ex}
For $n=2$, let the Grushin vector fields
$X=(\partial_{x_{1}},x_{1}\partial_{x_{2}})$ defined on an open
connected set $W\subset \mathbb{R}^2$. The Grushin operator induced
by $X$ is defined as
\begin{equation}\label{8-11}
  \triangle_{X}:=\frac{\partial^{2}}{\partial x_{1}^2}+x_{1}^2\frac{\partial^{2}}{\partial x_{2}^2}.
\end{equation}
Assume $\Omega\subset\subset W$ to be a bounded connected open subset of $W$ which has smooth and non-characteristic  boundary for $X$. Meanwhile $\Omega$ satisfies that $\Omega\cap\{(0,x_{2})|x_{2}\in \mathbb{R}\}\neq \varnothing $. Thus, there exists a point $(0,y_{0})\in \Omega$. Since $\Omega$ is an open set, then we can find $\delta>0$ such that $B((0,y_{0}),\delta)\subset \Omega$. It is easy to get that
\[ \sum|\det(Y_{i_{1}},Y_{i_{2}})|=|x_{1}|, \]
where the sum is over all 2-combinations $(Y_{i_{1}},Y_{i_{2}})$ of set $\{\partial_{x_{1}},x_{1}\partial_{x_{2}}\}$.
Observe that
\begin{align*}
\int_{\Omega}\frac{dx_{1}dx_{2}}{|x_{1}|}&\geq \int_{B((0,y_{0}),\delta)}\frac{dx_{1}dx_{2}}{|x_{1}|}\\
&=\int_{0}^{\delta}dr\int_{0}^{2\pi}\frac{1}{|\cos\theta|}d\theta=4\delta\int_{0}^{\frac{\pi}{2}}\frac{1}{\sin\theta}d\theta=  +\infty.
\end{align*}
Therefore, the vector fields do not satisfy the condition (A). However, by calculating directly, we have
\[ \Lambda((x_{1},x_{2}),r)=2(|x_{1}|r^2+r^{3}).\]
From \eqref{7-8}, we obtain
\[ \int_{\Omega}h_{D}(x,x,t)dx\leq C_{5}\int_{\Omega}\frac{1}{\Lambda((x_{1},x_{2}),\sqrt{t})}dx_{1}dx_{2}=
C_{7}\int_{\Omega}\frac{1}{|x_{1}|t+t^{\frac{3}{2}}}dx_{1}dx_{2},~\mbox{for all}~ 0<t<\delta_{1},\]
where $C_{7}=\frac{1}{2}C_{5}>0$. Since $\Omega\subset [a,b]\times\Omega_{x_{2}}$ for some $a<0<b$, we can deduce that
\begin{align*}
\int_{\Omega}\frac{1}{|x_{1}|t+t^{\frac{3}{2}}}dx_{1}dx_{2}&\leq \int_{[a,b]\times\Omega_{x_{2}}}\frac{1}{|x_{1}|t+t^{\frac{3}{2}}}dx_{1}dx_{2}\\
&=|\Omega_{x_{2}}|\int_{a}^{b}\frac{1}{|x_{1}|t+t^{\frac{3}{2}}}dx_{1}\\
&=|\Omega_{x_{2}}|\left(\int_{a}^{0}\frac{1}{|x_{1}|t+t^{\frac{3}{2}}}dx_{1}+\int_{0}^{b}\frac{1}{|x_{1}|t+t^{\frac{3}{2}}}dx_{1} \right)\\
&=|\Omega_{x_{2}}|\left(\int_{0}^{-a}\frac{1}{x_{1}t+t^{\frac{3}{2}}}dx_{1}+\int_{0}^{b}\frac{1}{x_{1}t+t^{\frac{3}{2}}}dx_{1} \right)\\
&=\frac{|\Omega_{x_{2}}|}{t}\left(\log\left(1-\frac{a}{\sqrt{t}}\right)+\log\left(1+\frac{b}{\sqrt{t}}\right) \right).
\end{align*}
Therefore, we have
\[ ke^{-\lambda_{k}t}\leq \sum_{j=1}^{k}e^{-\lambda_{j}t}\leq \int_{\Omega}h_{D}(x,x,t)dx\leq C_{7}\frac{|\Omega_{x_{2}}|}{t}\left(\log\left(1-\frac{a}{\sqrt{t}}\right)+\log\left(1+\frac{b}{\sqrt{t}}\right) \right)\]
holds for some $0<t<\delta_{1}$. Observe that if we take $t=\frac{1}{\lambda_{k}}$, then there exists $j_{0}\in \mathbb{N}^+$ large enough, such that $t<\delta_{1}$ for $k\geq j_{0}$. Thus, we have
\[ ke^{-1}\leq C_{7}|\Omega_{x_{2}}|\lambda_{k}\cdot\left(\log(1-a\sqrt{\lambda_{k}})+\log(1+b\sqrt{\lambda_{k}})\right)~~\mbox{for all}~ k\geq j_{0}.\]
That means $\lambda_k\geq Ck(\log k)^{-1}>Ck^{\frac{2}{3}}$ for $k$ large enough. Here the generalized M\'{e}tivier index $\tilde{\nu}=3$.

\end{ex}

\section{A remark on  uniform the   H\"{o}rmander condition}
   In this part, we introduce the uniform version of H\"{o}rmander's condition which was defined in \cite{Kusuoka1987} and \cite{Kusuoka1988}.\par
   For the vector fields $X'=(Z_{1},Z_{2},\cdots,Z_{q})$ defined in $\mathbb{R}^n$, we denote $J=(j_{1},\cdots,j_{k})$ with $1\leq j_{i}\leq m$, $|J|=k$ is the length of $J$. Then the $k$ order commutator $Z_{J}$ is defined as
\[ Z_{J}=[Z_{j_{1}},[Z_{j_{2}},[Z_{j_{3}},\cdots,[Z_{j_{k-1}},Z_{j_{k}}]\cdots]]]. \]
   We say that the vector fields $X'=(Z_{1},Z_{2},\cdots,Z_{q})$ satisfy the uniform version of H\"{o}rmander's condition in $\mathbb{R}^n$ if there exists a positive integer $Q$ and a positive constant $\alpha$ such that
\begin{equation}\label{9-1}
  \inf_{\eta\in \mathbb{S}^{n-1}}\sum_{|J|\leq Q}\left \langle  Z_{J}I(x), \eta \right \rangle_{\mathbb{R}^{n}}^{2}\geq \alpha~~\mbox{for all}~~ x\in \mathbb{R}^n.
\end{equation}
Here $\left \langle  \cdot, \cdot \right \rangle_{\mathbb{R}^{n}}$ is the inner product in $\mathbb{R}^n$, $Z_{J}I(x)$ is the vector in $\mathbb{R}^n$ which corresponding to the differential operator $Z_{J}$.
\par
   Now, we assume that $X'=(Z_{1},Z_{2},\cdots,Z_{q})$ is an extension of $X=(X_{1},X_{2},\cdots,X_{m})$ in Proposition \ref{prop5-1} and satisfies the H\"{o}rmander's condition (H) in $\mathbb{R}^n$ with H\"{o}rmander's index $Q$. Moreover, we know that
\begin{equation}
\label{9-2}
 X'=\left\{
          \begin{array}{ll}
          (X_{1},X_{2},\cdots,X_{m},0,0,\cdots,0), & \hbox{in $\Omega_{1}$;} \\[2mm]
           (0,0,\cdots,0,\partial_{x_{1}},\partial_{x_{2}},\cdots,\partial_{x_{n}}) , & \hbox{in $\mathbb{R}^{n}\setminus \Omega_{2}$.}
          \end{array}
        \right.
\end{equation}
Then we have
\begin{proposition}
\label{prop9-1}
The vector fields $X'$ in \eqref{9-2} satisfies the uniform H\"{o}rmander's condition \eqref {9-1} in $\mathbb{R}^n$ for the positive integer $Q$ and some constant $\alpha>0$.
\end{proposition}
\begin{proof}
It is simple to see that for any $x\in \mathbb{R}^{n}\setminus \Omega_{2}$ and $\eta\in \mathbb{S}^{n-1}$, we have
\[ \sum_{|J|\leq Q}\left \langle  Z_{J}I(x), \eta \right \rangle_{\mathbb{R}^{n}}^2=|\eta|^2=1. \]
Therefore, it suffices to prove that
\begin{equation}\label{9-3}
  \inf_{\eta\in \mathbb{S}^{n-1}}\sum_{|J|\leq Q}\left \langle  Z_{J}I(x), \eta \right \rangle_{\mathbb{R}^{n}}^2\geq \alpha~~\mbox{for all}~~ x\in \overline{\Omega_{2}}
\end{equation}
holds for some $\alpha>0$. If the assertion would not hold, then for any $n\in \mathbb{N}$, there exists a sequence $\{x_{n}\}_{n=1}^{\infty}\subset \overline{\Omega_{2}}$ such that
\[ \inf_{\eta\in \mathbb{S}^{n-1}}\sum_{|J|\leq Q}\left \langle  Z_{J}I(x_n), \eta \right \rangle_{\mathbb{R}^{n}}^2<\frac{1}{n}.\]
Hence, we can find a sequence $\{\eta_{n}\}_{n=1}^{\infty}\subset \mathbb{S}^{n-1}$ such that
\begin{equation}\label{9-4}
  \sum_{|J|\leq Q}\left \langle Z_{J}I(x_{n}),\eta_{n}\right \rangle_{\mathbb{R}^n}^2<\frac{1}{n}~~\mbox{for all} ~~ n\geq 1.
\end{equation}
Since $(x_{n},\eta_{n})\in \overline{\Omega_{2}}\times \mathbb{S}^{n-1}$ and $\overline{\Omega_{2}}\times \mathbb{S}^{n-1}$ is a compact set, we can find a subsequence $(x_{n_{k}},\eta_{n_{k}})\to (x_{0},\eta_{0})\in \overline{\Omega_{2}}\times \mathbb{S}^{n-1}$ as $k\to +\infty$. Thus, we can deduce from \eqref{9-4} that
\begin{equation}\label{9-5}
  \sum_{|J|\leq Q}\left \langle Z_{J}I(x_{0}),\eta_{0} \right \rangle_{\mathbb{R}^n}^2=0.
\end{equation}
Now, let $Y_{j}=\sum_{k=1}^{n}a_{jk}(x)\partial_{x_{k}}(1\leq j\leq n)$ be arbitrary $n$ vector fields which are chosen from the set $\{Z_{J}||J|\leq Q\}$. It can be deduced from \eqref{9-5} that
\begin{equation}\label{9-6}
\sum_{j=1}^{n}\left\langle Y_{j}I(x_{0}),\eta_{0}\right\rangle_{\mathbb{R}^n}^2=0.
\end{equation}
Therefore, \eqref{9-6} implies $\det(Y_{1},Y_{2},\cdots,Y_{n})(x_{0})=0$, which means $Z_{1},Z_{2},\cdots,Z_{q}$ together with their commutators up to length $Q$ cannot span the tangent space $T_{x_{0}}(\mathbb{R}^{n})$ at the point $x_{0}$. This leads to a contradiction. Thus we have the conclusion of Proposition \ref{prop9-1}.
\end{proof}

\section*{Acknowledgments}
The first version of this paper was done when the first author visited 
the Max-Planck Institute for Mathematics in the Sciences, Leipzig during July-August 
of 2018 as a visiting professor. He would like to thank Professor J. Jost (Max-Planck Institute for Mathematics in the Sciences, Leipzig) for the invitation and  financial support.





\end{document}